\documentclass[12pt]{amsart}
\usepackage{graphicx}
\usepackage{amsmath, enumerate}
\usepackage{amsfonts, amssymb, amsthm, xcolor, stmaryrd}
\usepackage{tikz, tikz-cd, caption, tabu, longtable, mathrsfs, xtab}
 \usepackage{dsfont, cite, todonotes}
\usepackage[toc,page]{appendix}
\usepackage{hyperref}

\numberwithin{equation}{section}
\theoremstyle{plain}

\newtheorem{thm}{Theorem}[section]
\newtheorem{lemma}[thm]{Lemma}
\newtheorem{prop}[thm]{Proposition}

\theoremstyle{definition}

\newtheorem{exmp}[thm]{Example}

\theoremstyle{remark}
\newtheorem{rmk}[thm]{Remark}

\newcommand{\kzxz}[4]{\left(\begin{smallmatrix} #1 & #2 \\ #3 & #4\end{smallmatrix}\right) }

\newcommand{\lp}{\left (}
\newcommand{\rp}{\right )}
\newcommand{\smat}[4]{\left(\begin{smallmatrix}
                 #1 & #2\\
                 #3 & #4
\end{smallmatrix}\right)}
\newcommand{\pmat}[4]{\begin{pmatrix}
                 #1 & #2\\
                 #3 & #4
\end{pmatrix}}

 \newcommand{\Nm}{{\mathrm{Nm}}}
\newcommand{\KM}{{\mathrm{KM}}}
\newcommand{\PSL}{{\mathrm{PSL}}}
\newcommand{\Sh}{{\mathrm{Sh}}}
\newcommand{\SL}{{\mathrm{SL}}}
\newcommand{\SO}{{\mathrm{SO}}}
\newcommand{\Spin}{{\mathrm{Spin}}}

\newcommand{\tth}{\textsuperscript{th }}
\newcommand{\sgn}{\mathrm{sgn}}
\newcommand{\Tr}{\mathrm{Tr}}
\newcommand{\tr}{\mathrm{tr}}

\newcommand{\Aut}{{\mathrm{Aut}}}

\newcommand{\erfc}{{\mathrm{erfc}}}

\newcommand{\ef}{\mathfrak{e}}

\newcommand{\Ac}{{\mathcal{A}}}

\newcommand{\Dc}{{\mathcal{D}}}

\newcommand{\Fc}{{\mathcal{F}}}

 \newcommand{\Oc}{\mathcal{O}}
 \newcommand{\Qc}{\mathcal{Q}}

\newcommand{\Bb}{\mathbb{B}}
\newcommand{\Cb}{\mathbb{C}}

\newcommand{\Hb}{\mathbb{H}}

\newcommand{\Nb}{\mathbb{N}}
\newcommand{\Pb}{\mathbb{P}}
\newcommand{\Qb}{\mathbb{Q}}

\newcommand{\Rb}{\mathbb{R}}
\newcommand{\Zb}{\mathbb{Z}}
\newcommand{\ebf}{\mathbf{e}}

\newcommand{\zbar}{{\overline{z}}}
\newcommand{\tpsi}{{\tilde{\psi}}}

\newcommand{\half}{{\tfrac{1}{2}}}

\renewcommand{\Re}{\mathrm{Re}}
\renewcommand{\Im}{\mathrm{Im}}
\newcommand{\res}{\mathrm{res}}

\newcommand{\Iso}{{\mathrm{Iso}}}

\newcommand{\mds}{{S}}

\newcommand{\Div}{{\mathrm{Div}}}
\newcommand{\DD}{{\mathrm{D}}}
\newcommand{\leg}[2]{\left( \frac{#1}{#2} \right)}
\newcommand{\Mp}{{\mathrm{Mp}}}
\newcommand{\ind}{{\mathrm{ind}}}
\newcommand{\Ib}{{\mathbf{I}}}

\newcommand{\Ind}{\mathrm{Ind}}

\AtBeginDocument{%
  \def\MR#1{}
  }
\begin{document}
\title{Modularity of generating series of winding numbers}
\author[J. H. Bruinier, J. Funke, \"{O}. Imamoglu, Y. Li]{Jan H. Bruinier, Jens Funke, \"{O}zlem Imamo\={g}lu, Yingkun Li}

\address{Fachbereich Mathematik,
Technische Universit\"at Darmstadt, Schlossgartenstrasse 7, 64289 Darmstadt, Germany}
\email{bruinier@mathematik.tu-darmstadt.de}

\address{
Department of Mathematical Sciences, University of Durham, South Road, Durham DH1 3LE, UK
}
\email{jens.funke@durham.ac.uk}

\address{
Departement Mathematik,
ETH Z\"urich,
R\"amistrasse 101,
8092 Z\"urich,
Switzerland}
\email{ozlem@math.ethz.ch}

\address{Fachbereich Mathematik,
Technische Universit\"at Darmstadt, Schlossgartenstrasse 7, 64289 Darmstadt, Germany}
\email{li@mathematik.tu-darmstadt.de}


\date{\today}
\maketitle

\begin{abstract}
The Shimura correspondence connects modular forms of integral weights and half-integral weights.
One of the directions is realized by the Shintani lift, where the inputs are holomorphic differentials and the outputs are holomorphic modular forms of half-integral weight.
In this article, we generalize this lift to differentials of the third kind. As an application we obtain a modularity result concerning the generating series of winding numbers of closed geodesics on the modular curve. 
\end{abstract}

\section{Introduction}
For an arithmetic subgroup $\Gamma \subset \SL_2(\Rb)$, the quotient $M = \Gamma \backslash \Hb$ is a  (possibly non-compact) Riemann surface. Denote by $M^*$ the standard compactification of $M$. On this real surface, there is an ample supply of closed geodesics associated to indefinite binary quadratic forms. Given any differential 1-form $\eta$ on $M^*$, it is natural to pair this form with these geodesics, and study the generating series of these numbers. This was initiated by Shintani for holomorphic $1$-forms $\eta=f(z)dz$ coming from a holomorphic cusp form $f$ \cite{Shintani75}.
He viewed $M^*$ as a locally symmetric space associated to the orthogonal group of signature $(2, 1)$, and considered the integral
\begin{equation}  
\label{eq:I}
I(\tau, \eta) := \int_{M^*} \eta \wedge \Theta(\tau, z, \varphi) d\bar{z},
\end{equation}
where $\Theta(\tau, z, \varphi) d\bar{z}$ is a theta function valued in differential 1-forms.
The output is the generating series of cycle integrals of $f$, and also a modular form of half-integral weight corresponding to $f$ under the Shimura correspondence \cite{shimura73}.
This construction has been extended to much greater generality by Kudla and Millson \cite{KM87, KM90} for closed, rapidly decaying differential forms $\eta$ on locally symmetric spaces associated to orthogonal and unitary groups to produce holomorphic Siegel and Hermitian modular forms. These results can be considered as analogs of the classical result of Hirzebruch and Zagier \cite{HZ76}.

In this paper, we study the case when $\eta$ is a differential 1-form of the third kind on $M^*$.
The integral $I(\tau, \eta)$ now needs to be suitably regularized (see  \eqref{eq:regI}).
The output will be a real-analytic modular form of weight $3/2$, whose holomorphic part consists of cycle integrals of $\eta$. 
In particular, if $M^*$ has genus zero, then these cycle integrals are closely related to the winding number of the closed geodesics around the poles of $\eta$.
To see this, we take as an example $\Gamma = \SL_2(\Zb)$ and consider
\begin{equation}
  \label{eq:fw}
  \eta_{} := \frac{j'(z)}{j(z) - 1728} dz,
\end{equation}
where  $j: M^* \to \mathbb{P}^1(\Cb)$ is the Klein $j$-invariant.

For a positive discriminant $D$, let $\Qc_D$ be the set of integral, binary quadratic forms $[A, B, C]=Ax^2+Bxy+Cy^2$ of discriminant $D$, which has an action by $\Gamma$. One can associate a closed geodesic $c_X$ to each $X \in \Qc_D$, whose image $c(X)$ in $M^*$ is a closed curve depending only on the $\Gamma$-equivalent class of $X$.
Let $\Ind(X)$ be the winding number of $j(c(X))$ with respect to $1728$ and $\infty$ on $\mathbb{P}^1(\Cb)$, where we count with half multiplicity if $j(c(X))$ passes through 1728 or $\infty$.
As a special case of our main theorem, we show that the generating functions of certain averages of these winding numbers possess modular properties.
\begin{thm}
  \label{thm:scalar1}
Let $\Delta < 0$ be a fixed fundamental discriminant, $\chi_{\Delta}$ the associated genus character (see  \eqref{eq:genus}), and $L(s, \Delta)$ the Dirichlet $L$-function associated to the Dirichlet character $\lp\frac{\Delta}{\cdot} \rp$.
Then the generating series
\begin{equation}
  \label{eq:Wind_gen}
  G_\Delta(\tau) :=  L(0, \Delta) + \sum_{\begin{subarray}{c} d > 0 \\ d \equiv 0, 3 \bmod{4} \end{subarray}}
q^d
  \sum_{X \in \Gamma \backslash \Qc_{-d\Delta}} \chi_{\Delta}(X) \Ind(X) 
\end{equation}
is a mixed mock modular form in the sense of \cite{DMZ12} (see section \ref{subsec:Weilrep}) of weight $3/2$ on $\Gamma_0(4)$. Its shadow is the Siegel indefinite theta function
$$
\theta_\Delta(\tau) := \frac{1}{\sqrt{|\Delta|}} \sum_{A, B, C \in \Zb} \chi_\Delta([A, B, C]) (A + C) q^{-(B^2 - 4AC)/|\Delta|} e^{-4 \pi v (B^2 + (A - C)^2)/|\Delta|}.
$$
\end{thm}

In general, we can consider the higher level analogue, where $\Gamma$ and $\Qc_D$ are replaced by the congruence subgroup $\Gamma_0(N)$ and subset 
$$
\Qc_{N, D} := \{[A, B, C] \in \Qc_D: N \mid A\} \subset \Qc_D
$$
respectively. In this case, the period integrals of an arbitrary differential of the third kind $\eta$ (with residue divisor with integral coefficients) are often transcendental by results of Scholl, Waldschmidt, and W\"ustholz, see e.g. \cite{Scholl86, Wald79, Wu84}. Therefore the coefficients in the generating series above will in general not be algebraic. Nevertheless, the generating series can still be made modular by the following general result, of which Theorem \ref{thm:scalar1} is a special case.

\begin{thm}
  \label{thm:scalarN}
Let $N$ be a positive integer and $\Delta < 0$ a fundamental discriminant satisfying $\Delta \equiv r^2 \bmod{4N}$, and $\chi_\Delta$ the genus character as defined in \eqref{eq:genus}.
For an arbitrary differential of the third kind $\eta$ with residue divisor $\sum_{\zeta \in M^*} r_\zeta(\eta) \cdot [\zeta]$ and $X \in \Qc_{N, D}$, denote $\int_{c(X)} \eta$ the cycle integral of $\eta$ along the geodesic on $M^*$ determined by $X$ (see \eqref{eq:cycle_int}). We define the twisted trace of index $d$ for $\eta$ by 
\[
\Tr_{N, \Delta, d}(\eta) :=  \sum_{X \in \Gamma_0(N) \backslash \Qc_{N, -d\Delta}} \frac{\chi_{\Delta}(X)}{2\pi i} \int_{c(X)} \eta. 
\]
Then the generating series
\begin{equation}
  \label{eq:GN}
  G_{\Delta, N}(\tau) := \sum_{\begin{subarray}{c} d > 0 \\ d \equiv 0, 3 \bmod{4} \end{subarray}}
\Tr_{N, \Delta, d}(\eta) q^d
\end{equation}
becomes a real-analytic modular form of weight $3/2$ on $\Gamma_0(4N)$ after adding the function $\Theta^*_\Delta(\tau, \eta)$ in \eqref{eq:Theta*}, whose image under the lowering operator is the indefinite Siegel theta function $\Theta_\Delta(\tau, \eta)$ in \eqref{eq:ThetaDelta}.
\end{thm}

\begin{rmk}
Theorem~\ref{thm:scalar1} follows from Theorem~\ref{thm:scalarN} by specializing to $N=1$ and the residue divisor $[\infty]-[\sqrt{-1}]$ on the modular curve. 
\end{rmk}

Assume that the residue divisor in Theorem~\ref{thm:scalarN} is given by $[\zeta_1]-[\zeta_2]$, the difference of two points in $M^*$. Then the theta lift $I(\tau, \eta)$ is closely related to the indefinite theta functions $\vartheta^{c_1, c_2}$ constructed by Zwegers \cite{ZwThesis} and Kudla and the second author \cite{FK17} associated to two (arbitrary) preimages $c_1$ and $c_2$ of $[\zeta_1]$ and $[\zeta_2]$ in the extended upper half plane $\Hb^*$ under the quotient map $\Hb^* \to M^*$. Explicitly, $\vartheta^{c_1, c_2}$ can be realized by integrating the Kudla-Millson theta form over any curve in $\Hb^*$ connecting $c_2$ with $c_1$. Then the two non-holomorphic forms $I(\tau, \eta)$ and $\vartheta^{c_1, c_2}$ differ by a holomorphic cusp form, more precisely by a (cohomological) Shintani lift. In particular, $I(\tau, \eta)$ and  $\vartheta^{c_1, c_2}$ coincide if the genus of $M^*$ is zero.

The paper is organized as follows. First, we recall the basic setup in Section \ref{sec:setups}. Then we will prove the main result (Theorem \ref{thm:FE}), and deduce Theorem \ref{thm:scalarN} from it. After that we relate the integral $I(\tau, \eta)$ to the indefinite theta function of Zwegers and Funke-Kudla. Finally in the last section, we specialize Theorem \ref{thm:FE} to the case of a Shimura curve and give an example related to Ramanujan's mock theta function.
\vspace{.2in}

\noindent \textbf{Acknowledgement.} Most of this work was done during visits by J.F. to TU Darmstadt, by J.F. and Y.L. to FIM in Z\"{u}rich in 2017. We thank these organizations for their supports. J.B and Y.L. are supported by the DFG grant BR-2163/4-2. Y.L. is also supported by an NSF postdoctoral fellowship.

\section{Setup}
\label{sec:setups}
\subsection{Modular Curve.}
\label{subsec:modcurve}
Let $B$ be an indefinite quaternion algebra over $\Qb$ and let $V = B^0 \subset B$ be the subspace of elements of trace zero.
For $N$ a positive rational number, there is a quadratic form $Q(\alpha):= -N \Nm(\alpha)$ on $V$, where $\Nm$ is the reduced norm on $B$. 
This turns $V$ into a quadratic space of signature $(2, 1)$ and $G = \Spin(V)$ is an algebraic group over $\Qb$. Denote by $\bar{G}$ the image of $G$ in $\SO(V)$.

Since $B$ is indefinite, the real points $V_\Rb := V \otimes_\Qb \Rb$ of $V$ are isomorphic to 
$$
M_2(\Rb)^0 := \left\{
\pmat{x_1}{x_2}{x_3}{-x_1}: x_1, x_2, x_3 \in \Rb
\right\},
$$
where the norm on $B$ is identified with the determinant and the action of $\Spin(V)(\Rb)$ becomes that of $\SL_2(\Rb)$ via 
$$
g\cdot X := g X g^{-1},
$$
for $X \in M_2(\Rb)^0$. We can identify the symmetric space associated to $V$ with the Grassmannian of oriented negative lines in $V_\Rb$, and we denote by $\Dc$ its connected component containing the line spanned by $ \frac{1}{\sqrt{N}}\smat{}{1}{-1}{}$. Then $\Dc$ can be identified with the upper half-plane $\Hb$ as follows. For $z = x + iy \in \Hb$, let $g_z = \smat{1}{x}{}{1} \smat{y^{1/2}}{}{}{y^{-1/2}} \in \SL_2(\Rb)$ be the element such that $g_z \cdot i = z$. 
Then $X(z) := \frac{1}{\sqrt{N}} g_z \cdot \smat{}{1}{-1}{} = \frac{1}{\sqrt{N} y} \smat{-x}{|z|^2}{-1}{x}$ spans a negative line in $V_\Rb$. It is easy to check that $\gamma \cdot X(z) = X(\gamma z)$.
Furthermore, we denote
\begin{align}
  W(z) &:= \partial_\zbar X(z) = \frac{1}{2 i \sqrt{N} y^2} \pmat{-z}{z^2}{-1}{z}; \\
R(X, z) &:=  (X, X) + \frac{1}{2} (X, X(z))^2.
\end{align}
Let $(\cdot, \cdot)_z$ be the majorant associated to $z$; a small calculation gives
\begin{equation}
  \label{eq:majorant}
    (X, X)_z = (X, X) + (X, X(z))^2.  
\end{equation}
Note that the quantity $(X, X(z))$ never vanishes for nonzero $X$ when $Q(X) \leq 0$.

For $L \subset V$ an even, integral lattice with quadratic form $Q$ and dual lattice $L'$, let $\Gamma_L \subset G(\Qb)$ be the discriminant kernel and $\Gamma \subset \Gamma_L \cap \SO^+(L)$ the connected component that fixes the component $\Dc$.
The quotient $M = M_L := \Gamma \backslash \Dc$ is a modular curve or Shimura curve. The latter occurs exactly when  
$B$ does not split over $\Qb$, that is, $\Iso(V)$, the set of rational isotropic lines in $V$, is empty. If $M$ is non-compact, then $M$ can be compactified by adding $\Gamma \backslash \Iso(V)$. In that case we denote $\Dc^* := \Dc \cup \Iso(V)$ and $M^* = M^*_L := \Gamma \backslash \Dc^*$.

\begin{exmp}
\label{ex:N}
Suppose that $V(\Qb) = M_2(\Qb)^0$. Then $\SL_2 \cong \Spin(V)$ as algebraic groups over $\Qb$. There is also a bijection between $\Pb^1(\Qb)$  and $\Iso(V)$ by sending $(\alpha:\beta)$ to the line spanned by $\smat{-\alpha \beta}{\alpha^2}{-\beta^2}{\alpha \beta}$.
This identifies $\Dc^*$ with $\Hb^* := \Hb \cup \Pb^1(\Qb)$, where the action of $\SO(V)$ on $\Dc^*$ becomes the linear fractional transformation on $\Hb^*$.
Consider the lattice
$$
L := \left\{
\smat{-B}{-C/N}{A}{B} \in V : A, B, C \in \Zb
\right\} \subset V.
$$
Its dual lattice is given by 
$$
L' := \left\{
\pmat{-B/2N}{-C/N}{A}{B/2N} \in V : A, B, C \in \Zb
\right\}.
$$
Then $\Gamma=\Gamma_0(N)$ and $\SO^+(L)$ is the extension of $\Gamma_0(N)$ by all Atkin-Lehner involutions, see Section 2.4 of \cite{BO10}.

\end{exmp}

\subsection{Cusps and Fundamental Domains.}
\label{subsec:cusp}
Suppose that $V$ splits over $\Qb$, i.e.\ $V(\Qb) = M_2(\Qb)^0$.
Let $\ell_0 \in \Iso(V)$ be the line spanned by $u_0 = \smat{0}{1}{0}{0}$, which corresponds to the cusp $i \infty$. 
For any $\ell \in \Iso(V)$, let $\sigma_\ell \in \SL_2(\Zb)$ be an element such that $\ell = \sigma_\ell \ell_0$. Set $u_\ell := \sigma_\ell^{} u_0$. If $\Gamma_\ell \subset \Gamma$ is the stabilizer of $\ell$, then 
\begin{equation}
\label{eq:sigmal}
\sigma_\ell^{-1} {\Gamma}_{\ell} \sigma_\ell = \left\{ \pm \pmat{1}{k\alpha_\ell}{}{1}: k \in \Zb \right\}
\end{equation}
with $\alpha_\ell \in \Zb_{>0}$ the width of the cusp $\ell$. Choose $\beta_\ell \in \Qb_{>0}$ such that $\beta_\ell u_\ell$ is a primitive element of $\ell \cap L$, and write $\varepsilon_\ell := \alpha_\ell/\beta_\ell$.
For $h \in L'/L$, we can also define $h_\ell \in \Qb_{\ge 0}$ to be the quantity that satisfies $0 \le h_\ell < \beta_\ell$ and $\ell \cap (L + h) = (\Zb\beta_\ell + h_\ell) u_\ell$.
We use $[\ell]$ to represent the class of $\ell$ in the finite set $\Gamma \backslash \Iso(V)$.
Note that $\alpha_\ell, \beta_\ell, \varepsilon_\ell$ and $h_\ell$ are all class invariants.

For $\epsilon > 0$ and $z \in M$, let $B_\epsilon(z)$ be the open $\epsilon$-neighborhood with respect to the hyperbolic metric. 
For $\ell \in \Iso(V)$, we also have an open neighborhood $B_\epsilon(\ell)$ near $\ell$ given by $q_\ell^{-1}(U_{\epsilon^{1/\alpha_\ell}})$, where $U_r :=\{w \in \Cb: |w| < r \}$ and $q_\ell(z) = \ebf(\sigma_\ell^{-1} z / \alpha_\ell)$.
These charts turns $M^* := \Gamma \backslash \Dc^*$ into a compact Riemann surface.

For a finite subset $\mds = \mds^\circ \cup \mds_\infty \subset M^*$ with $\mds^\circ \subset M$ and $\mds_\infty \subset \Gamma \backslash \Iso(V)$, let $M^*_\mds$ be the Riemann surface with the points in $\mds$ removed. 
Suppose $\epsilon > 0$ is sufficiently small such that $B_\epsilon(\zeta)$ are all disjoint for $\zeta \in \mds$. Then
\begin{equation}
  \label{eq:MSe}
M^*_{\mds, \epsilon} := M^* \backslash \lp \bigsqcup_{\zeta \in \mds} B_\epsilon(\zeta) \rp.
\end{equation}
is a compact Riemann surface with boundary $ \partial M^*_{\mds, \epsilon} = - \bigsqcup_{\zeta \in \mds} \partial B_\epsilon(\zeta) $.

Suppose $S = S_\infty = \Gamma \backslash \Iso(V) \subset M^*$, then $M = M^*_S$ and the fundamental domain of $M^*_{S, \epsilon}$ can be described as follows.
For $T > 1$, let $\Fc_T := \{ z \in \Hb: |\Re(z)| \le 1/2 \text{ and } T \ge |z| \ge 1\}$ be the standard, truncated fundamental domain of $\SL_2(\Zb)$.
For $\alpha \in \Nb$, denote 
$$
\Fc^\alpha_T := \bigcup_{j = 0}^{\alpha - 1} \pmat{1}{j}{}{1} \Fc_T.
$$
Then for $\epsilon > 0$ sufficiently small, the fundamental domain of $M_{S, \epsilon}^*$ is given by
$$
\Fc_T(\Gamma) := \bigcup_{\ell \in \Gamma \backslash \Iso(V)} \sigma_\ell \Fc_{T}^{\alpha_\ell}
$$
where $T = - \log \epsilon/2\pi$.

\subsection{Differentials of the third kind.}

\label{subsec:differential}
Let $X$ be a connected, non-singular projective curve over $\Cb$ of genus $g$.
A \textit{differential form of the third kind} is a meromorphic differential 1-form $\eta$ on $X$ with at most simple poles and residues in $\Qb$. 
Denote the set of its singularities by
\begin{equation}
  \label{eq:Sing}
  \mds(\eta) := \{ \zeta \in X: \eta \text{ has a simple pole at }\zeta \},
\end{equation}
and the residue of $\eta$ at $z \in X$ by $r_z(\eta)$. 
Given a real-analytic function $h$ with no poles in the $\epsilon_0$-neighborhood $B_{\epsilon_0}(z)$ for sufficiently small $\epsilon_0 > 0$, we have
\begin{equation}
  \label{eq:Cauchy}
  h(z) r_z(\eta) = \frac{1}{2\pi i} \lim_{\epsilon \to 0} \int_{\partial B_\epsilon(z)} h  \eta.
\end{equation}
This is true for all $\epsilon > 0$ sufficiently small when $h$ is holomorphic.
The \textit{residue divisor} of $\eta$ is defined by
\begin{equation}
  \label{eq:div}
\res(\eta) := \sum_{\zeta \in \mds(\eta)} r_\zeta(\eta) \cdot [\zeta] \in \Div^0(X)
\end{equation}
and has degree zero by the residue theorem.
Conversely, for any divisor $\DD \in \Div^0(X)$, there is a differential of the third kind $\eta_\DD$ such that $\res(\eta_\DD) = \DD$ \cite{Griffiths89, BO10}.
Furthermore, such a differential is uniquely determined up to addition of holomorphic differentials.
In fact, by the Riemann period relations, there exists a \textit{unique} differential of the third kind $\eta_\DD$ for each $\DD \in \Div^0(X)$ such that
$$
\Re \int_{\gamma } \eta_\DD = 0
$$
for every smooth curve $\gamma \subset X \backslash \DD$.
We call $\eta_\DD$ the \textit{canonical differential of the third kind}.

For a finite set $S \subset X$, denote $X' := X \backslash S$. The homology group $H_1(X', \Zb)$ is a finitely generated free abelian group, and there is a canonical surjective homomorphism
\begin{equation}
\label{eq:surj}
\pi: H_1(X', \Zb) \twoheadrightarrow H_1(X, \Zb)
\end{equation}
induced by $X' \hookrightarrow X$. The kernel is isomorphic to $\Zb^{|S|}$ and generated by the class of $\partial B_{\epsilon}(\zeta)$ in $H_1(X', \Zb)$ for each $\zeta \in S$ and $\epsilon > 0$ sufficiently small. We denote these classes by $\mathbf{c}_\zeta \in H_1(X', \Zb)$.
For each class $\mathbf{c}' \in \ker(\pi)$, we can define the winding number of $\mathbf{c}'$ around $\zeta \in S$ to be the integer $\ind(\mathbf{c}'; \zeta)$ satisfying
\begin{equation}
  \label{eq:Wind}
  \frac{1}{2\pi i} \int_{\mathbf{c}'} \eta = \sum_{\zeta \in S} \ind(\mathbf{c}' ; \zeta) \cdot r_\zeta(\eta),  
\end{equation}
for any differential of the third kind $\eta$ satisfying $S(\eta) \subset S$.

To define the winding number of any $\mathbf{c}' \in H_1(X', \Zb)$, we need to fix a section $s$ of the surjective homomorphism \eqref{eq:surj}.
Then we can define $\ind_s(\mathbf{c}'; \zeta)$ for any $\mathbf{c}' \in H_1(X', \Zb)$ and $\zeta \in S$ by
\begin{equation}
  \label{eq:Wind2}
  \ind_s(\mathbf{c}'; \zeta) := \ind(\mathbf{c}' - s \circ \pi(\mathbf{c}'); \zeta).
\end{equation}
This quantity clearly depends on the section $s$ and satisfies
\begin{equation}
  \label{eq:homology}
  \mathbf{c}' - s \circ \pi( \mathbf{c}') = \sum_{\zeta \in S} \ind_s(\mathbf{c}'; \zeta) \mathbf{c}_\zeta.
\end{equation}
in $H_1(X', \Zb)$.

For $\zeta_1, \zeta_2 \in S$, let $\gamma \subset X$ be any curve going from $\zeta_1$ to $\zeta_2$ and $\Ib(\gamma, \mathbf{c}')$ the signed intersection of $\gamma$ with any $\mathbf{c}' \in H_1(X', \Zb)$. 
It can be defined as the integral of the Poincar\'{e} dual of $\mathbf{c}' \in H_1(X', \Zb)$ along $\gamma$.
It is clear that $\Ib(\gamma, \mathbf{c}')$ only depends on the class of $\gamma$ in $H_1(X', \partial X', \Zb)$.
From the fundamental polygon of a Riemann surface, we can see that for every section $s$ and pairs of points $\zeta_1, \zeta_2 \in S$, there exists a curve $\gamma_s$ such that $\Ib(\gamma_s, s(\mathbf{c})) = 0$ for all $\mathbf{c} \in H_1(X, \Zb)$.

For each differential $\eta$ of the third kind with $S(\eta) \subset S$, there exists (by Poincar\'{e} duality) a closed differential $\omega_{\eta, s}$ on $X$ unique in $H^1_{\mathrm{dR}}(X)$ such that
\begin{equation}
\label{eq:omega_eta}
\int_{s(\mathbf{c})} \eta = \int_{\mathbf{c}} \omega_{\eta, s}
\end{equation}
 for all $\mathbf{c} \in H_1(X, \Zb)$.
For any $\mathbf{c}' \in H_1(X', \Zb)$, we have
\begin{equation}
  \label{eq:Wind3}
  \frac{1}{2\pi i} \int_{\mathbf{c}'} \eta - \omega_{\eta, s} = \sum_{\zeta \in S} 
\ind_s(\mathbf{c}'; \zeta) \cdot r_\zeta(\eta).
\end{equation}
Using  \eqref{eq:homology}, we can relate the intersection number to this integral for certain $\eta$.
\begin{lemma}
  \label{lemma:intersection}
Suppose $\res(\eta) = \zeta_1 - \zeta_2$. 
In the notation above, we have
\begin{equation}
  \label{eq:Wind_intersection}
  \frac{1}{2\pi i} \int_{\mathbf{c}'} \eta - \omega_{\eta, s} = 
-\Ib(\gamma_s, \mathbf{c}')
\end{equation}
for any $\mathbf{c}' \in H_1(X', \Zb)$. 
\end{lemma}

\subsection{Geodesics and Cycle Integrals}
\label{subsec:cycle_int}
An element $X \in V(\Qb)$ of positive norm $m$ defines a geodesic $c_X \subset \Dc$ via
$$
c_X := \{ z \in \Dc: z \perp X\} = \{ z \in \Dc: (X, X(z)) = 0\}.
$$
Suppose $X = \smat{-B/2}{-C}{A}{B/2} $. Then $c_X$ is explicitly given by $ \{z \in \Hb^*: A|z|^2 + B\Re(z) + C = 0 \}$. The geodesics carry a natural orientation gives as follows. The semicircle $c_X$ is oriented counter-clockwise if $A> 0$ and clockwise if $A<0$. The stabilizer $\bar{\Gamma}_X \subset \bar{\Gamma}$ of $X$ is either infinite cyclic or trivial. The latter happens if and only if $\sqrt{m/N} \in \Qb^\times$.
 In that case, we call $X$ \textit{split-hyperbolic}.
The orthogonal complement $X$ in $V$ is spanned by two rational isotropic lines $\ell_X, \tilde{\ell}_X$ and $c_X$ connects the two corresponding cusps of $\Dc$.
We distinguish them by requiring $\ell_X$ to be the endpoint of $c_X$.
Since the action of $\Gamma$ preserves the orientation of $c_X$, we have 
$\ell_{\gamma \cdot X} = \gamma \cdot \ell_X$ for all $\gamma \in \Gamma$.
Note that $\tilde{\ell}_X = \ell_{-X}$.
Choose $\sigma_{\ell_X} \in \SL_2(\Zb)$ such that
\begin{equation}
\label{eq:sigmaX}
\sigma^{-1}_{\ell_X} X = \sqrt{m/N} \pmat{1}{-2r_{X}}{}{-1}
\end{equation}
for a unique $r_{X} \in \Qb \cap [-\alpha_{\ell_X}/2, \alpha_{\ell_X}/2)$. We call $r_X$ the real part of $c_X$, since the geodesic $c_X$ is given by
$
c_X = \sigma_{\ell_X} \{ z \in \Dc: \Re(z) = r_{X}\}.
$
We denote the image of $c_X$ on $M^*$ by $c(X)$, which is a closed geodesic cycle when $X$ is not split-hyperbolic.
For $S \subset M^*$ finite and  $\epsilon > 0$ sufficiently small, we write 
$$
c(X; \mds,  \epsilon) := c(X) \cap M^*_{\mds, \epsilon}
$$
which is a compact subset of $M^*_\mds$.
For convenience, we also set
 \begin{equation}
   \label{eq:iota}
\iota(X) := 
\begin{cases}
  1, & X \text{ is split-hyperbolic,} \\
0, & \text{ otherwise.}
\end{cases}
 \end{equation}

Let $\eta$ be a differential form of the third kind on $M^*$ and $\mds := \mds(\eta) \cup (\Gamma \backslash \Iso(V))$.
Then we can define
\begin{equation}
  \label{eq:cycle_int}
\int_{c(X)} \eta := \text{p.v. } \lim_{\epsilon \to 0} 
\lp
 \int_{c(X; \mds, \epsilon)} \eta  
-
 \iota(X) ( r_{\ell_X}(\eta) - r_{\ell_{-X}}(\eta))  \frac{ \log \epsilon}{2\pi i}
  \rp,
\end{equation}
where we take the Cauchy principal value of the limit. It exists by the proposition below. 
\begin{prop}
  The right hand side of  \eqref{eq:cycle_int} exists.
\end{prop}
\begin{proof}
When $\iota(X) = 0$, $c(X) \subset M$ is a closed geodesic. 
If $\mds \cap c(X) = \emptyset$, then the limit clearly exists.
Otherwise, we can understand the right hand side as the follows.
Suppose $\zeta \in \mds^\circ(\eta) \cap c(X) \subset M$, then for any $\epsilon > 0$,  $c(X) \cap \partial B_\epsilon(\zeta)$ has two points, which can be connected along $\partial B_\epsilon(\zeta)$ in two ways. We can keep $c(X)$ closed by deforming it near $\zeta$ in these two ways.
Each will yield a value as $\epsilon  \to 0$. The Cauchy principal value of $\lim_{\epsilon \to 0} \int_{c(X; \mds(\eta), \epsilon) \cap B_{\epsilon_0}(\zeta)} \eta$ is then the average of these two values.
Here $\epsilon_0$ is chosen small enough such that the $B_{\epsilon_0}(z)$'s are disjoint for $z \in\mds$.

When $\iota(X) = 1$, the geodesic $c(X)$ connects two cusps $\ell = \ell_X$ and $\ell_{-X}$ on $M^*$. 
Near $\zeta \in S^\circ(\eta) \cap c(X)$, the same argument for $\iota(X) = 0$ works.
Near the cusp $\ell$, we have $\eta = ( r_\ell(\eta) + O(e^{-C y_\ell})) dz_\ell$ in the local coordinate $z_\ell = x_\ell + i y_\ell =  \sigma_\ell^{-1} z$ by the definition of $r_\ell(\eta)$.
Therefore, we have
$$
\int_{c(X) \cap (B_{\epsilon_0}(\ell) - B_\epsilon(\ell))} \eta = r_\ell(\eta) \int^{-\log \epsilon /2\pi}_{-\log \epsilon_0/2\pi} i dy_\ell + O(\epsilon)
=  r_\ell(\eta) \frac{\log \epsilon}{2\pi i} + O(\epsilon)
$$
near the cusp $\ell = \ell_X$. The same argument with the reversed orientation also works for $\ell_{-X}$ and finishes the proof.
\end{proof}

\section{Theta series and the main result}

\subsection{Weil representation and modular forms.}
\label{subsec:Weilrep}

Let ${\Mp}_2(\Rb)$ be the metaplectic two-fold cover of $\SL_2(\Rb)$ consisting of elements $(M,  \phi(\tau))$ with $ M = \smat{a}{b}{c}{d} \in \SL_2(\Rb)$ and $\phi: \Hb \to \Cb$ a holomorphic function satisfying $\phi(\tau)^2 = c\tau + d$. 
Let $\Gamma' \subset {\Mp}_2(\Rb)$ be the inverse image of $\SL_2(\Zb) \subset \SL_2(\Rb)$ under the covering map. 
It is generated by $T := (\smat{1}{1}{0}{1}, 1)$ and $S = (\smat{0}{-1}{1}{0}, \sqrt{\tau})$, where $\sqrt{\cdot}$ denotes the principal branch of the holomorphic square root.

Let  $L$ be an even lattice with quadratic form $Q$ and associated bilinear form $(\cdot, \cdot)$ and dual lattice $L'$ of signature $(b^+,b^-)$. The (finite) Weil representation $\rho_L$  of $\Gamma'$ associated to $L$ acts on  $\Cb[L'/L]$ (see e.g., \cite{Shintani75,Borcherds98}) and is given by 
\begin{equation}
  \label{eq:Weil_rep}
    \rho_L(T)(\ef_h) := \ebf(Q(h)) \ef_h,  \qquad
\rho_L(S)(\ef_h) := \frac{\ebf(-(b^+ - b^-)/8)}{\sqrt{|L'/L|}} \sum_{\mu \in L'/L} \ebf(-(h, \mu)) \ef_\mu.
\end{equation}
Here we denote the standard basis of $\Cb[L'/L]$ by $\{\ef_h: h \in L'/L\}$, Furthermore, we write $\langle \cdot, \cdot \rangle$ for the standard hermitian scalar product on $\Cb[L'/L]$. Note that if $L = L_1 \oplus L_2$ then 
 $\Cb[L'/L]$ can be identified with $\Cb[L_1'/L_1] \otimes \Cb[L_2'/L_2]$ by sending $\ef_{(h_1, h_2)}$ to $\ef_{h_1} \otimes \ef_{h_2}$ for $(h_1, h_2) \in L'_1/L_1 \oplus L'_2/L_2 = L'/L$, and we have $\rho_{L} = \rho_{L_1} \otimes \rho_{L_2}$.

 A function $f: \Hb \to \Cb[L'/L] $ is called \textit{modular} with weight $k \in \half \Zb$, representation $\rho_L$ and level $\Gamma \subset \Gamma'$ if
\begin{equation}
  \label{eq:mod}
(  f \mid_k (M, \phi))(\tau) := \phi(\tau)^{-2k} f (M \cdot \tau) = \rho_L((M, \phi)) f(\tau)
\end{equation}
for all $(M, \phi(\tau)) \in \Gamma$. Let $\Ac_k(\Gamma, \rho_L)$ denote the space of such functions that are real-analytic. It contains the usual subspaces $M_k(\Gamma, \rho_L), S_k(\Gamma, \rho_L)$ of holomorphic modular and cusp forms respectively, and also $H_k(\Gamma, \rho_L)$, the space of harmonic (weak) Maass forms, see \cite{BF04}. 
We drop $\Gamma$, resp.\ $\rho_L$, if it is $\Gamma'$, resp.\ trivial.
In case $\Gamma \subset \Mp_2(\Rb)$ is the double cover of $\Gamma'' \subset \SL_2(\Zb)$, we also write $\Ac_k(\Gamma'', \rho_L)$ for $\Ac_k(\Gamma, \rho_L)$. The same notation holds for any subspace.

In \cite{Za09}, Zagier defined the notion of a \textit{mock modular form} following Zwegers \cite{ZwThesis}. 
For a cusp form $g(\tau)  \in S_{2-k}({\overline{\rho_L}})$, define
\begin{equation}
  \label{eq:Eichler}
  g^*(\tau) := \frac{i}{2} \int^{i \infty}_{- \overline{\tau}} \frac{\overline{g({-\zbar})}}{(-i(z + \tau)/2)^{k}} dz.
\end{equation}
A holomorphic function $f: \Hb \to \Cb[L'/L]$ is called \textit{mock modular} of weight $k$ with respect to $\rho_L$ if there exists a $g \in S_{2-k}(\overline{\rho}_L)$ such that $\tilde{f} := f + g^* \in H_{k}(\rho_L)$. The form $g$ is called the \textit{shadow} of $f$. The shadow $g$ is related to $f + g^*$ via the operator $\xi_k := 2iv^k \overline{\partial_{\overline{\tau}}}$ in \cite{BF04}, i.e., $\xi_k( f(\tau) + g^*(\tau))  = {g({\tau})}$.

The notion of mock modular form was further generalized in \cite{DMZ12} to mixed mock modular forms. 
We say that a holomorphic function $f$ is \textit{mixed mock modular} if there exists $g \in S_{2-k + \ell}(\rho_{L_1})$ and $h \in M_{\ell}(\rho_{L_2})$ such that $f + h \otimes g^* \in \Ac_k(\Gamma, \rho_L)$. In this case, the function $g(\tau) \otimes \overline{h({\tau})}$ is called the \textit{shadow} of $f$.

\subsection{Unary theta series associated to a cusp}

Suppose $b^+ = 2, b^- = 1$ and $L$ contains isotropic vectors.
For an isotropic line $\ell \in \Iso(V)$, the space $\ell^\perp/\ell$ is a one-dimensional positive definite space with quadratic form $Q$. It contains an even lattice 
$$
K_\ell := (L \cap \ell^\perp)/(L \cap \ell),
$$
whose dual lattice is given by $K'_\ell := (L' \cap \ell^\perp)/(L' \cap \ell)$.
There is an exact sequence
$$
(L' \cap \ell)/(L \cap \ell) \to (L' \cap \ell^\perp )/(L \cap \ell^\perp) \to K'_\ell /K_\ell.
$$
For $h \in L'/L$ with $h \perp \ell$, let $\overline{h}$ denote its image in $K'_\ell/K_\ell$.
Define a unary theta function $\Theta_\ell$ by
\begin{equation}
  \label{eq:Thetal}
 \Theta_\ell(\tau) := \sum_{h \in L'/L, \; h \perp \ell} \ef_h \sum_{X \in K_\ell + \overline{h}} (X, \Re(W(\sigma_\ell i))) \ebf(Q(X) \tau).
\end{equation}
It transforms with weight $3/2$ and the Weil representation $\rho_L$.
Furthermore, it only depends on the class of $\ell$ in $\Gamma \backslash \Iso(V)$. The Fourier coefficients take  the following shape.
\begin{lemma}
  \label{lemma:FCl}
For $m \in \Qb_{> 0}$ and $h \in L'/L$, let $b_\ell(m, h)$ be the $(m, h)$\tth Fourier coefficient of $\Theta_\ell(\tau)$. Then 
\begin{equation}
  \label{eq:FCl}
  b_\ell(m, h) = -\frac{\sqrt{N}}{2\varepsilon_\ell}  \sum_{X \in \Gamma_\ell \backslash (L_{m, h} \cap \ell^\perp)} \delta_\ell(X),
\end{equation}
where $L_{m, h} := \{ \lambda \in L + h: Q(\lambda) = m\}$ and $\delta_\ell(X) = \pm 1$ if $\ell = \ell_{\pm X}$. 
\end{lemma}

\begin{proof}
If $h$ is not orthogonal to $\ell$, then both sides above are zero trivially.
So suppose $h \perp \ell$. Then there is a natural surjection from the set $L_{m, h}\cap \ell^\perp$ to $\{X \in K_\ell + \overline{h}: Q(X) = m\}$. The kernel is the subgroup of $L' \cap \ell$ that acts on $L_{m, h} \cap \ell^\perp$ by addition, which is exactly $L \cap \ell$. 
The map $L_{m, h} \cap \ell^\perp \to (L_{m, h} \cap \ell^\perp) / (L \cap \ell)$ now factors through $L_{m, h} \cap \ell^\perp \to \Gamma_\ell \backslash (L_{m, h} \cap \ell^\perp)$.
Arguing as in Lemma 3.7 of \cite{Funke04}, we know that 
$$
 \Gamma_\ell \backslash (L_{m, h} \cap \ell^\perp) \to 
(L_{m, h} \cap \ell^\perp) / (L \cap \ell)
$$
is an $2\sqrt{m/N} \varepsilon_\ell$ to 1 covering map.
For an element $X \in L_{m, h} \cap \ell^\perp$ with $\ell_X = \ell$,  \eqref{eq:sigmaX} implies that
 $$
(X, \Re(W(\sigma_\ell i))) = 
 \frac{1}{2 \sqrt{N}} \lp \sigma_\ell^{-1} X, \pmat{-1}{0}{0}{1} \rp = - \sqrt{m},
$$
with the sign changed if $\ell_{-X} = \ell$ instead. This  finishes the proof.
\end{proof}

\subsection{Schwartz Forms and Theta functions.}

Recall that the metaplectic group $\Mp_2(\Rb)$ acts on the space of Schwartz functions on $V_\Rb$ via the Weil representation $\omega$ over $\Rb$. For a convenient reference, see e.g. \cite{Shintani75}. We consider the Shintani Schwartz function $\varphi_\Sh$ \cite{Shintani75}, 
\begin{equation} \label{eq:Shintani}
 \varphi_{\Sh}(X, z) := (X, W(z)) e^{-\pi(X, X)_z}. %
\end{equation}
Then $\varphi_{\Sh}(X, z)$ has weight $3/2$, that is, 
\[
\omega(k'(\theta)) \varphi_{\Sh}= \chi_{3/2}(k'(\theta)) \varphi_{\Sh}. 
\]
Here $k'(\theta)$ is a preimage of $\kzxz{\cos\theta}{\sin{\theta}}{-\sin{\theta}}{\cos\theta} \in \SO(2)$ in $\Mp_2(\Rb)$ and $\chi_{3/2}$ is the character of $K'$ whose square is given by  $\chi_{3/2}^2(k'(\theta)) = e^{3 i\theta}$. 
The Kudla-Millson Schwartz form \cite{KM82} is closely related to $\varphi_{\Sh}$ and is given by 
\begin{equation}
  \label{eq:KM_form}
  \varphi_\KM(X, z) :=  \varphi_\Sh(X, z) d \overline{z} + \overline{\varphi_\Sh(X, z)} dz.
\end{equation}
Then $\varphi_\KM(X)$ defines a closed $1$-form on $\Dc$. 
The associated ``Millson" Schwartz function \cite{KM90},
\begin{equation}
  \label{eq:psi}
    \psi_{}(X, z) := \frac{(X, X(z))}{2} e^{-\pi (X, X)_z}
\end{equation}
has weight $-1/2$. In what follows, it will be convenient to set
\[
\varphi^0(X):= \varphi(X) e^{\pi (X,X)},
\]
where $\varphi$ is any of the functions above. For a Schwartz function $\varphi$ of weight $\ell$, we set
\[
\varphi(X,\tau,z):= v^{-\ell/2}\omega(g'_\tau)\varphi(x)= v^{3/4-\ell/2} \varphi^0(\sqrt{v}X,z) e^{\pi i (X,X)\tau}.
\]
Here, $g'_{\tau} = \left( \kzxz{u^{1/2}} {v^{-1/2}v}{0}{u^{-1/2}}, u^{-1/4}\right) \in \Mp_2(\Rb) $ maps the base point $i \in \Hb$ to $\tau=u+iv \in \Hb$. Finally note that all functions (and forms) above are $G(\Rb)$-equivariant, that is,
\[
\varphi(g\cdot X,g\cdot z)= \varphi(X,z) \qquad \qquad (g \in G(\Rb)).
\]
In particular, for fixed $X$, the above functions are $\Gamma_X$-invariant. 

For $\varphi \in \{ \varphi_\Sh, \varphi_\KM, \psi\}$ and $L \subset V$ a lattice as in Section \ref{subsec:modcurve}, we define the theta series
\begin{equation}
  \label{eq:Theta}
\Theta(\tau, z, \varphi) := \sum_{h \in L'/L} \Theta_h(\tau, z, \varphi) \ef_h, \; \qquad \Theta_h(\tau, z, \varphi) :=  \sum_{X \in L + h} \varphi(X, \tau, z).
\end{equation}
The series all converge absolutely due to the exponential decay of $\varphi$. Then $\Theta(\tau, z, \varphi_\Sh)$ and $\Theta(\tau, z, \psi)$ transform in $\tau$ with respect to the Weil representation $\rho_L$ with weight $3/2$ and $-1/2$ respectively, see \cite{Shintani75}. As a function of $z$, $\Theta(\tau, z, \varphi_\Sh)$ and $\Theta(\tau, z, \psi)$ have weights $2$ and $0$ for $\Gamma_L$ respectively, while $\Theta(\tau, z, \varphi_\KM)$ defines a closed differential $1$-form on $M$.

When $m \in \Qb^\times$, the set $\Gamma \backslash L_{m, h}$ is finite, where $L_{m, h}$ is defined in Lemma \ref{lemma:FCl}.
Therefore, we can write the Fourier expansion of $\Theta_h(\tau, z, \varphi)$ as (see \cite[Equation (4.15)]{BF06})
\begin{equation}\label{eq:decomp0}
 \Theta_h(\tau, z, \varphi) =
\Theta_{0,h}(\tau, z, \varphi)
+  \sum_{m \in \Qb^\times} \ebf(m\tau) \sum_{X \in \Gamma \backslash L_{m, h}} a(\sqrt{v}X, z, \varphi^0),
\end{equation}
where 
\begin{equation}\label{eq:decomp1}
\Theta_{0,h}(\tau, z, \varphi) := \sum_{X \in L_{0, h}} \varphi^0(\sqrt{v} X, z), \qquad
a(X, z, \varphi^0) := \sum_{\gamma \in \Gamma_X \backslash \Gamma}  \varphi^0(\gamma^{-1} X, z). 
\end{equation}

Using a partial Poisson summation, we can describe their behaviors near the cusps of $M$ as follows (see \cite{FM02}). 
\begin{prop}
\label{prop:growth}
Let $z_\ell = x_\ell + i y_\ell := \sigma_\ell^{-1} z $ be the local coordinate at the cusp $\ell \in \Iso(V)$.
As $y_\ell \to \infty$, we have
\begin{align}
  \label{eq:growth}
  \Theta(\tau, z, \varphi_\Sh) &= \frac{1}{ \sqrt{N} \beta_\ell} \Theta_\ell(\tau) + O(e^{-C y_\ell^2 \min(v, 1/v)}), \\
  \Theta(\tau, z, \psi) &= O(e^{-C v y_\ell^2}) \notag
\end{align}
for some absolute constant $C > 0$ depending only on $L$. Furthermore, let $X \in L_{m, h}$ be a hyperbolic element, i.e.\ $m > 0$. Then 
\begin{align}
\label{eq:abound1}
 a(\sqrt{v}X, z, \varphi^0_\Sh) = 
                           \begin{cases}
                             \lp \mp \frac{1}{2 \alpha_\ell} + O(e^{-C y^2_\ell \min\{mv, 1/(mv)\}}) \rp  & \mathrm{if} \; \iota(X) = 1, [\ell] = [\ell_{\pm X}], \\
O(e^{-C mv y^2_\ell}), & \mathrm{otherwise.}
                           \end{cases}
 \end{align}
Here $\alpha_\ell$ is the width of the cusp $\ell$ as defined in Section~\ref{subsec:cusp}.
\end{prop}

We can extend $\Theta(\tau, z, \varphi_\Sh)$ and $\Theta(\tau, z, \psi)$ to $M^*$ by continuity.

\subsection{A singular Schwartz function}

We define a singular (Schwartz) function $\tpsi$ on $V_\Rb$ by
\begin{equation} \label{eq:tpsi}
\tpsi (X, z) := - \frac{\sgn(X, X(z))}{2} \erfc(\sqrt{\pi} |(X, X(z))|) e^{-\pi (X, X)}, \\
\end{equation}
where $\erfc(t) := \frac{2}{\sqrt{\pi}} \int^\infty_t e^{-r^2} dr$ is the complementary error function. We set $\tpsi(0)=0$. Note that $\tpsi (X, z)$ has a singularity along the geodesic $c_X$ if $Q(X)>0$ and is smooth otherwise. We also set $\tpsi(X)=0$ along $c_X$. As before we write 
$\tpsi^0(X, z) =\tpsi (X, z) e^{\pi (X,X)}$ and set
\[
\tpsi (X, \tau, z) = \tpsi^0 (\sqrt{v} X, z)  \ebf(Q(X) \tau)
\]
as if $\tpsi$ had weight $3/2$ under the Weil representation (which of course it doesn't). This is motivated by the following Lemma~\ref{lemma:diff_eq}(ii) which can be obtained by a direct calculation (such as in \cite{FK17}).

\begin{lemma} \label{lemma:diff_eq}

\begin{itemize}

\item[(i)]
Away from the singularity of $\tpsi(X, z)$, 
\begin{equation}
\label{eq:primitive}
 \overline{\partial} \tpsi(X, z) =  \varphi_\Sh^0(X, z)  d\overline{z}.
\end{equation}
Here $ \overline{\partial} $ is the exterior anti-holomorphic derivative for $z\in \mathcal{D}$. In particular, for $Q(X) \le 0$, the function $\tpsi^0(X, z)$ is a smooth $\overline{\partial}$-primitive of 
$ \varphi_\Sh^0(X, z)  d\overline{z}$. For $Q(X)<0$, the function $ \tpsi^0(X, z)$ satisfies the following current equation, 
 \begin{equation}
  \label{eq:diff_eq}
  \int_{\Gamma_X \backslash \Dc}  \left( \overline{\partial} \eta\right) \wedge \tpsi^0(X, z) =  - \int_{c_X} \eta + \int_{\Gamma_X \backslash \Dc} \eta \wedge \varphi_\Sh^0(X, z)  d\overline{z}.
\end{equation}
Here $\eta$ is a compactly supported $1$-form on $\Gamma_X \backslash \Dc$. 

\item[(ii)]
For the action of $L_{\tau} := -2iv^2 \partial_{\overline{\tau}}$, the weight-lowering operator in $\tau$, we have
\begin{equation}
L_{\tau} \tpsi(X, \tau, z)  =  \psi(X, \tau, z).
\end{equation}
\end{itemize}
\end{lemma}

We do not introduce the ``full" theta series $\Theta(\tau, z, \tpsi)$ since it would yield a function with dense singularities in $\mathcal{D}$. However, the individual Fourier coefficients define much better behaved functions with locally finite singularities. We set 
\[
\Theta_{m,h}(v, z, \tpsi) :=  \sum_{X \in L_{m,h}} \tpsi^0(\sqrt{v}X, z) \qquad \qquad (m \in \Qb)
\]
and also write if $Q(X) \ne 0$, $a(X, z, \tpsi^0) := \sum_{\gamma \in \Gamma_X \backslash \Gamma}  \tpsi^0(\gamma^{-1} X, z)$ for the individual $\Gamma$-orbits. 

We now state the behavior of $\Theta_{m,h}(v, z, \tpsi)$ at the cusps, which will be proved in the next subsection. We let $\Bb_1(x)$ be the first periodic Bernoulli polynomial given by
\begin{equation}
  \label{eq:B1}
  \Bb_1(x)  := x - (\lceil x \rceil + \lfloor x \rfloor)/2.
\end{equation}

\begin{lemma}
  \label{lemma:abound}
Let $X \in L_{m, h}$ be a hyperbolic element, i.e.,\ $m > 0$, and let $\ell \in \Iso(V)$ be a cusp with local coordinate $z_\ell := x_\ell + i y_\ell = \sigma_\ell^{-1} z$.
The function $a(\sqrt{v}X, z, \tpsi^0)$ has the following asymptotic expansions at $\ell$,
\begin{align}
 \label{eq:abound2}
 & a(\sqrt{v}X, z, \tpsi^0) = 
                           \begin{cases}
                             \mp \Bb_1 \lp \frac{x_\ell - r_{\pm X}}{\alpha_\ell} \rp + O(e^{-C y^2_\ell \min\{mv, 1/(mv)\}}), & \mathrm{if} \; \iota(X) = 1, [\ell] = [\ell_{\pm X}], \\
O(e^{-C mv y^2_\ell}), & \mathrm{otherwise.}
                           \end{cases}
\end{align}
Here $r_X$ is defined in \eqref{eq:sigmaX}. Let $\beta_\ell, h_\ell$ be the quantities defined in Section \ref{subsec:cusp}. Then 
  \begin{align*}
\Theta_{0,h}(\tau, z, \tpsi) &= - \Bb_1 \lp \frac{h_\ell}{\beta_\ell} \rp + O(e^{-Cv y_\ell^2}) \qquad \qquad (y_\ell \to \infty).
  \end{align*}

\end{lemma}

We set for all $h \in L'/L$ and $ [\ell] \in  \Gamma \backslash \Iso(V)$, 
\begin{align}\label{eq:tpsicusp}
   \Theta_{m,h}(v, [\ell], \tpsi) 
:= 
\begin{cases}   
  0 \qquad & m\ne0, \\
  - \Bb_1 \lp \frac{h_\ell}{\beta_\ell} \rp   \quad & m=0.
 \end{cases}  
 \end{align}

\subsection{Boundary Behavior of $\tpsi$}\label{Boundary-tpsi}

For $\Re(s) > -1/2$ and $\kappa > 0$, define
\begin{equation}
  \label{eq:g}
  g(w; \kappa, s) :=
  \frac{\pi^{-s - 1/2} |w|^{-2s} \sgn(w)}{2} \Gamma \lp \pi \kappa w^2, s + \frac{1}{2} \rp.
\end{equation}
Note that $g(w; \kappa, 0) = \frac{\sgn(w)}{2} \erfc(\sqrt{\pi \kappa} |w|)$ and $g(w; \kappa, 1/2) = \frac{e^{-\pi w^2 \kappa}}{2 \pi w}$.
We are interested in the periodic function
\begin{equation}
  \label{eq:G}
  G(x; \kappa, s) := \sum_{n \in \Zb} g(x + n; \kappa, s),
\end{equation}
which converges absolutely and uniformly on compact subsets of $\Rb \backslash \Zb$ for $\Re(s) > -1/2$. Using Poisson summation, we can calculate its Fourier expansion.
\begin{lemma}
\label{lemma:FT}
At $s = 0$, the function $G(x; \kappa, s)$ is bounded, 1-periodic in $x$ and has the expansion
\begin{equation}
  \label{eq:GFe}
  G(x; \kappa, 0) =   -\Bb_1(x) + i \sum_{m \in \Zb, m \neq 0} g(m, \kappa^{-1}, 1/2) \ebf(mx),
\end{equation}
where the sum above converges absolutely and uniformly.
\end{lemma}

\begin{proof}
It is easy to see that the function $h(x, s) := G(x; \kappa, s) + \Bb_1(x)$ is bounded, odd and 1-periodic on $\Rb$ for $0 \ge  \Re(s) > -1/2$. 
Therefore, we just need to calculate its $m$\tth Fourier coefficients for $m \neq 0$.
In this case, we start with
$$
\int_0^1 h(x, s) \ebf(-mx) dx = \widehat{g}(-m; \kappa, s) - \frac{1}{2 \pi i m},
$$
which becomes the $m$\tth Fourier coefficient of $h(x, 0)$ after taking $s \to 0$ from the left by the dominated convergence theorem. A standard calculation gives us
$$
\widehat{g}(\widehat{w}; \kappa, s) := \int_{-\infty}^\infty g(w;\kappa, s) \ebf(w \widehat{w}) dw = i
\lp \frac{\sgn(\widehat{w})}{|\widehat{w}|^{1-2s}} \frac{\Gamma(1-s)}{2 \pi^{1-s}} -
  g(\widehat{w}, \kappa^{-1}, 1/2 - s) \rp
$$
for $-1/2 < \Re(s) < 1$.
Taking $ s \to 0$ and using the oddness of $g$ then gives us  \eqref{eq:GFe}.
\end{proof}

\begin{proof}[Proof of Lemma~\ref{lemma:abound}]
  Suppose $\iota(X) = 1$ and $\ell = \ell_X$.  Then $m = Nk^2$ for some $k \in \Qb_{> 0}$ and $\bar{\Gamma}_X$ is trivial. 
Write $\sigma_{\ell}^{-1} X = k \smat{1}{-2r}{}{-1}$ with $r = r_X \in \Qb$ as in  \eqref{eq:sigmaX}, $\alpha = \alpha_{\ell}$.
Arguing as in Lemma 5.2 of \cite{BF06} gives us
\begin{align*}
a(\sqrt{v}X, z, \tpsi^0) &= \sum_{\gamma \in \Gamma_{\ell}} \tpsi^0(\sqrt{v} X, \gamma z) + O(e^{-C mv y_\ell^2}).
\end{align*}
Using  \eqref{eq:sigmal}, we can rewrite
$$
\sum_{\gamma \in \Gamma_{\ell}} \tpsi^0(\sqrt{v} X, \gamma z)
= \sum_{n \in \Zb} \tpsi^0(\sqrt{vm/N} \smat{1}{2(-r + \alpha n)}{}{-1}, z_\ell) = G \lp \frac{x_\ell - r}{\alpha}; \frac{4mv\alpha^2}{y_\ell^2}, 0 \rp.
$$
Applying Lemma \ref{lemma:FT} then finishes the proof for $\ell = \ell_X$. The case with $\ell = \ell_{-X}$ is similar.
\end{proof}

The asymptotic behavior of $\Theta_{0,h}(\tau, z, \tpsi)$ can be calculated analogously.

\subsection{Theta Integral and Main Theorem.}
\label{subsec:input}
In this section, we will state the main theorem, which evaluates the integral in  \eqref{eq:I} explicitly when $\eta$ is a meromorphic {differential of the third kind} on $M^*$. 
For a finite set $S \subset M^*$ containing $S(\eta)$, we define
\begin{equation}
\label{eq:regI}
I(\tau, \eta) := \lim_{\epsilon \to 0} 
\lp 
\int_{M^*_{S, \epsilon}}  \eta \wedge \Theta(\tau, z, \varphi_\Sh)   d\zbar
+  \sum_{\ell \in \Gamma \backslash \Iso(V)} \frac{r_{\ell}(\eta) \varepsilon_\ell}{\sqrt{N} }  \Theta_{\ell}(\tau) \frac{\log \epsilon}{ \pi i } \rp.
\end{equation}
Note that the second term is present only when $\eta$ is not rapidly decaying at the cusps, i.e., $S_\infty(\eta) \neq \emptyset$.
This agrees with the truncation in the usual parameter $T = - \frac{\log \epsilon}{2\pi}$.
\begin{prop}
  \label{prop:exist}
  The limit in \eqref{eq:regI} exists and defines a real-analytic modular form that transforms with respect to the Weil representation $\rho_L$ of weight $3/2$.
\end{prop}

\begin{proof}
  Since $\eta$ has only simple poles on $M$, the limit in $T$ exists by Prop.\ \ref{prop:growth}.
  Since the integral over $M^*_{\mds, \epsilon}$ is modular of weight $3/2$ as a function in $\tau$ for any $\epsilon > 0$ sufficiently small, so is $I(\tau, \eta)$.
  \end{proof}

For $m \in \Qb$ and $h \in L'/L$, we define the $m$\tth trace of $\eta$ by
\begin{equation}
  \label{eq:trace}
  \Tr_{m, h}(\eta) :=
  \begin{cases}
    \frac{1}{2\pi i} \sum_{X \in \Gamma \backslash L_{m, h}}    \int_{c(X)} \eta, & m > 0, \\
0, & \text{otherwise.}
  \end{cases}
\end{equation}
The main result is as follows.
\begin{thm}\label{thm:FE}
Let $L$ be an lattice as in Section \ref{subsec:modcurve} and $\eta$ be a meromorphic {differential of the third kind} on $M^* = M^*_L$ with residue divisor $\sum_{\zeta \in M^*} r_\zeta(\eta) \cdot [\zeta]$. Then the components of the vector-valued modular form $I(\tau, \eta) \in \Ac_{3/2}(\rho_L)$ have the expansion
   \begin{align*}
\frac{1}{2\pi i} I_h(\tau, \eta)&= 
\sum_{m \in \Qb_{>0}}  \Tr_{m, h}(\eta) q^m 
 - \sum_{\zeta \in M^*}  r_{\zeta}(\eta) \sum_{m \in \Qb} \Theta_{m,h}(v, \zeta,\tpsi^0)q^m.
\end{align*}
Under the lowering operator $L_{\tau}$, we have 
\[
L_\tau I(\tau, \eta) = - 2\pi i \Theta(\tau, \res(\eta), \psi).
\]
\end{thm}

\subsection{Orbital integrals and Proof of Main Theorem}
\label{sec:OrbBound}
Let $\eta$ be a fixed meromorphic 1-form of the third kind and set $\mds := \mds(\eta) \cup (\Gamma \backslash \Iso(V)) \subset M^*$.

We first consider the non-constant coefficients. Let $X \in \Gamma \backslash L_{m, h}$ with $m \neq 0$ and consider for $\epsilon > 0$,  the orbital integral
 \begin{equation}
  \label{eq:orb_int}
I_X(v, \eta; \mds, \epsilon) :=  \int_{M^*_{\mds, \epsilon}}  \eta \wedge a(\sqrt{v} X, z, \varphi^0_\Sh) d\zbar,
\end{equation}
where $a(X, z, \varphi^0_\Sh)$ is defined in  \eqref{eq:decomp1}. We directly see

\begin{prop}
\label{prop:Stokes}
The integral defining $I_X(v, \eta; \mds, \epsilon)$ converges for all $m \in \Qb^\times$, and even termwisely when $\sqrt{Nm} \not\in \Qb$.
Furthermore, we have
\begin{equation}
  \label{eq:Stokes}
  I_X(v, \eta; \mds, \epsilon) = 
 \int_{c(X; \mds, \epsilon)} \eta + \int_{\partial M^*_{\mds, \epsilon}} 
a(\sqrt{v} X, z, \tpsi^0)  \eta.
\end{equation}
\end{prop}

\begin{proof}
The convergence follows from the same argument in \cite{BF06}.
The rest is simply an application of Lemma \ref{lemma:diff_eq} and Stokes' theorem.
\end{proof}
\begin{prop}
\label{prop:orbit1}
Suppose $X \in L_{m, h}$ with $m \neq 0$. Then
\begin{equation}
\label{eq:orb1}
\lim_{\epsilon \to 0}
\lp I_X(v, \eta; \mds, \epsilon)  -
\iota(X) ( r_{\ell_X}(\eta) - r_{\ell_{-X}}(\eta))  \frac{\log \epsilon}{2\pi i}\rp =
 \left( \int_{c(X)} \eta \right)- 2\pi i a(\sqrt{v} X, \res(\eta), \tpsi^0),
\end{equation}
where $\iota(X)$ and $\int_{c(X)} \eta$ are defined in  \eqref{eq:iota} and \eqref{eq:cycle_int} respectively.
\end{prop}

\begin{proof}
We first subtract $\iota(X)( r_{\ell_X}(\eta) - r_{\ell_{-X}}(\eta))  \frac{\log \epsilon}{2\pi i}$ from both sides of  \eqref{eq:Stokes}. Taking the limit $\epsilon \to 0$ on the right hand side gives $2\pi i a(\sqrt{v} X, \res(\eta), \tpsi^0) - \int_{c(X)} \eta$ by  \eqref{eq:Cauchy}, \eqref{eq:cycle_int} and  $ \partial M^*_{\mds, \epsilon} = - \bigsqcup_{\zeta \in \mds} \partial B_\epsilon(\zeta) $ for $\epsilon > 0$ small.

Now for $\zeta \in M$ and $\epsilon > 0$ sufficiently small, the integral $\int_{B_\epsilon(\zeta)} \eta \wedge a(\sqrt{v} X, z, \varphi^0_\Sh) d \zbar$ converges since $\eta$ has at most a simple pole at $\zeta$.
If $\iota(X) = 0$, then $c(X)$ is compact in $M$ and the limit in $\epsilon$ exists.
Therefore, it suffices to consider $\iota(X) = 1$ and show that the following limit exists
\begin{equation}
\label{eq:cusp_reg}
\lim_{\epsilon \to 0} \int_{\Fc_T(\Gamma) - \Fc_{T_0}(\Gamma)} \eta \wedge a(\sqrt{v} X, z, \varphi^0_\Sh) d\zbar
- 
( r_{\ell_X}(\eta) - r_{\ell_{-X}}(\eta))  \frac{\log \epsilon}{2\pi i},
\end{equation}
where $T = -\log \epsilon / 2\pi$ and $T_0 > 1$ is a fixed constant depending on $\eta$ and $ \Gamma$. 
We can now decompose $\Fc_T(\Gamma) - \Fc_{T_0}(\Gamma) = \bigsqcup_{[\ell] \in \Gamma \backslash \Iso(V)} \sigma_\ell (\Fc_T - \Fc_{T_0})$.
Recall that $z_\ell = x_\ell + i y_\ell = \sigma_\ell^{-1} z$ and $dz_\ell \wedge d\overline{z_\ell} = -2i dx_\ell \wedge dy_\ell$.
By Lemma \ref{lemma:abound} below, we have
$$
\int_{\Fc_T - \Fc_{T_0}} \eta \wedge a(\sqrt{v} X, z_\ell, \varphi^0_\Sh) d \overline{z_\ell}
= 
\mp \frac{1}{2 \alpha_\ell} \int^T_{T_0} \int_{0}^{\alpha_\ell} r_\ell(\eta)  (-2i) dx_\ell dy_\ell + O(1) = \pm i r_\ell(\eta) T + O(1)
$$
when $[\ell ] = [\ell_{\pm X}]$, and $O(1)$ otherwise. Adding these over all $[\ell] \in \Gamma \backslash \Iso(V)$ shows that the limit in  \eqref{eq:cusp_reg} exists.
\end{proof}

By the decomposition \eqref{eq:decomp0}, this takes care of the non-constant coefficients of $I(\tau, \eta)$ since 
$$
-\sum_{\ell \in \Gamma \backslash \Iso(V)} \frac{2r_\ell(\eta)\varepsilon_\ell}{\sqrt{N}} \Theta_\ell(\tau)
=
\sum_{h \in L'/L} \ef_h \sum_{m \in \Qb^\times} \ebf(m\tau) \sum_{X \in \Gamma \backslash L_{m ,h}} \iota(X)(r_{\ell_X}(\eta) - r_{\ell_{-X}}(\eta)). 
$$
Indeed, this is a direct consequence of Lemma \ref{lemma:FCl} as
\begin{align*}
-\sum_{[\ell] \in \Gamma \backslash \Iso(V)} \frac{2r_\ell(\eta)\varepsilon_\ell b_\ell(m, h)}{\sqrt{N}} 
&=  
\sum_{[\ell] \in \Gamma \backslash \Iso(V)}
\sum_{X \in \Gamma_\ell \backslash (L_{m, h} \cap \ell^\perp)} r_\ell(\eta) \delta_\ell(X) \\&=
\sum_{X \in \Gamma \backslash L_{m ,h}} \iota(X)(r_{\ell_X}(\eta) - r_{\ell_{-X}}(\eta)).
\end{align*}
From the definition of $I(\tau, \eta)$ in \eqref{eq:regI}, the constant term is given by 
$$
I_{0,h}(\tau, \eta) := \lim_{\epsilon \to 0} 
\int_{M^*_{S, \epsilon}}  \eta \wedge
\Theta_{0,h}(\tau, z, \varphi_\Sh)d\zbar.
$$
Applying Stokes' Theorem as in Prop.\ \ref{prop:Stokes} and Lemma \ref{lemma:abound} we obtain
\begin{align*}
I_{0,h}(\tau, \eta) &=  \lim_{\epsilon \to 0} 
\int_{\partial M^*_{S, \epsilon}}\Theta_{0,h}(\tau, z, \tpsi) \eta \\ &=
-  \sum_{\zeta \in M} r_{\zeta}(\eta)\sum_{m \in \Qb} \Theta_{0,h}(v, \zeta,\tpsi^0)  -  \lim_{\epsilon \to 0}  \sum_{[\ell] \in \Gamma \backslash \Iso(V)} \int_{\partial B_\epsilon([\ell])} \Theta_{0,h}(\tau, z, \tpsi) \eta 
\\&=
-  \sum_{\zeta \in M} r_{\zeta}(\eta)\sum_{m \in \Qb} \Theta_{0,h}(v, \zeta,\tpsi^0) +
\sum_{[\ell] \in \Gamma \backslash \Iso(V)} r_{[\ell]}(\eta) \Bb_1 \lp \frac{h_\ell}{\beta_\ell} \rp\\
&= -  \sum_{\zeta \in M^*} r_{\zeta}(\eta)\sum_{m \in \Qb} \Theta_{0,h}(v, \zeta,\tpsi^0).
\end{align*}
This completes the proof of Theorem \ref{thm:FE}.

\subsection{Proof of Theorem \ref{thm:scalarN}.}
For $N \in \Nb$, let $L\subset V$ be the lattice in Example \ref{ex:N}. Then $\Gamma = \Gamma_0(N) = \Gamma_L \cap \SO^+(L)$ and we have a bijection
\begin{equation}
\label{eq:bijection}
\begin{split}
    \bigcup_{h \in L'/L} L_{m, h} & \to \Qc_{N, 4Nm}\\
\smat{-B/(2N)}{-C/N}{A}{B/(2N)} & \mapsto [NA, B, C],
\end{split}
\end{equation}
where $\Qc_{N, 4Nm}$ is defined as in the introduction.
Note that for any $f(\tau) = \sum_{h \in L'/L} f_h(\tau) \in \Ac_k(\rho_L)$, the component $f_0(\tau)$ is in $\Ac_{k}(\Gamma_0(4N))$.
Applying the Fricke involution to $f_0(\tau)$ and the transformation formula of $\rho_L(S)$ tells us that $\mathrm{sc}(f(4N\tau)) \in \Ac_k(\Gamma_0(4N))$, where
\begin{equation}
  \label{eq:sc}
  \begin{split}
    \mathrm{sc}: \Cb[L'/L] & \to \Cb\\
\sum_{h \in L'/L} a_h \ef_h & \mapsto \sum_{h \in L'/L} a_h.
  \end{split}
\end{equation}

Fix a fundamental discriminant $\Delta < 0$ and $r \in \Zb$ such that $\Delta \equiv r^2 \pmod{4 N}$. 
We follow \cite{BO10} to define the genus character $\chi_\Delta$ on $X = \smat{- B/2N}{-C/N}{A}{B/2N} \in L'$ by 
\begin{equation}
\label{eq:genus}
\chi_\Delta(X) = \chi_{\Delta}([NA, B, C]) :=
\begin{cases}
   \leg{\Delta}{n},  & \text{if } \Delta \mid B^2 - 4NAC \text{ and } (B^2 - 4NAC)/\Delta \text{ is }\\
& \text{a square modulo } 4N \text{ and } \gcd(A, B, C , \Delta) = 1,\\
0, & \text{ otherwise,}
\end{cases}
\end{equation}
where $n$ is any integer prime to $\Delta$ represented by one of the quadratic forms $[N_1 A, B, N_2 C]$ with $N_1N_2 = N$ and $N_1, N_2 > 0$ (see also \cite[\textsection 1.2]{GKZ87} and \cite[\textsection 1]{Sko90}).
It is invariant under the action of $\Gamma_0(N)$ and all Atkin-Lehner operators.

Let $L_\Delta := (\Delta L, \frac{Q}{|\Delta|})$ the scaled lattice. Then $L'_\Delta = L'$ and there is a natural projection map $\pi: L'/L_\Delta \to L'/L$. The following group 
\begin{equation}
  \label{eq:GammaDelta}
\Gamma_\Delta := \Gamma_{L_\Delta} \cap \mathrm{SO}^+(L_\Delta)  \subset \Gamma = \Gamma_0(N)
\end{equation}
 has finite index and contains the congruence subgroup $\Gamma(|\Delta|) \cap \Gamma_0(N|\Delta|)$.
The linear map 
\begin{equation}
  \label{eq:psichi}
  \begin{split}
  \psi_{\Delta, r} : \Cb[L'/L] & \to \Cb[L'/L_\Delta] \\
\ef_h & \mapsto \frac{1}{[\Gamma: \Gamma_\Delta]} \sum_{\begin{subarray}{c} \delta \in L'/L_\Delta \\ \pi(\delta) = rh \\ \frac{Q(\delta)}{\Delta} \equiv Q(h) \bmod{\Zb}\end{subarray}} \chi_\Delta(\delta) \ef_\delta
  \end{split}
\end{equation}
intertwines the representations $\rho_L$ and $\rho_{L_\Delta}$ \cite{AE13}.
With respect to the pairings on $\Cb[L'/L]$ and $\Cb[L'/L_\Delta]$, we obtain a linear map $\psi^*_{\Delta, r}: \Ac_{k}(\rho_{L_\Delta}) \to \Ac_k(\rho_L)$.

Now, we can apply Theorem \ref{thm:FE} to the lattice $L_\Delta$. 
From the bijection \eqref{eq:bijection}, we have
$$
\mathrm{sc} \circ\psi_{\Delta, r}^* \lp \sum_{\delta \in L'/L_\Delta} \Tr_{m, \delta}(\eta) \ef_\delta\rp
= \Tr_{N, \Delta, 4Nm}(\eta)
$$
for $m > 0$, where $\Tr_{N, \Delta, d}$ was defined in \eqref{eq:GN}.
Similarly, the Fourier coefficients of the non-holomorphic part becomes
$$
\mathrm{sc} \circ\psi_{\Delta, r}^* \lp \sum_{\delta \in L'/L_\Delta} \Theta_{m, \delta}(v, \zeta, \tpsi^0) \ef_\delta\rp
 = 
- \sum_{X \in \Qc_{N, 4N |\Delta| m}}  \frac{ \chi_\Delta(X) \sgn(\mathrm{d}(X, \zeta))}{2 [\Gamma: \Gamma_\Delta]}
\erfc \lp  \sqrt{ \frac{ \pi v}{N |\Delta| }} |\mathrm{d}(X, \zeta)| \rp
$$
for all nonzero $m \in \Qb$, where we have denoted
\begin{equation}
  \label{eq:delta}
 \mathrm{d}(X, \zeta) := \frac{A|\zeta|^2 + B \Re(\zeta) + C}{ \Im(\zeta)}
\end{equation}
for $\zeta \in \Hb$ and $X = [A, B, C] \in \Zb^3$.
When $m = 0$, the coefficient $\Theta_{m, h}(v, \zeta, \tpsi^0) = 0$ unless $\zeta = [\ell]$ is a cusp, in which case
$$
\mathrm{sc} \circ\psi_{\Delta, r}^* \lp \sum_{\delta \in L'/L_\Delta} \Theta_{0, \delta}(v, [\ell], \tpsi^0) \ef_\delta\rp
 = 
- \frac{1}{[\Gamma: \Gamma_\Delta]} \sum_{\delta \in L'/L_\Delta} \chi_\Delta(\delta) \Bb_1 \lp \frac{\delta_\ell}{|\Delta| \beta_\ell}  \rp.
$$
The quantities $\chi_\Delta(\delta)$ and $\beta_\ell$ only depends the $\Gamma_0(N)$-class of $\delta$ and $\ell$. Even though $\delta_\ell$ depends on the choice of $\ell$, the averaged quantity above does not since $\Gamma_0(N)/\Gamma_\Delta$ induces an automorphism on $L'/L_\Delta$. 
If $I(\tau, \eta)$ is the integral defined in \eqref{eq:regI} for the lattice $L_\Delta$, then Theorem \ref{thm:FE} implies that 
$$
I_{\Delta, N}(\tau, \eta) := \mathrm{sc}\circ \psi^*_{\Delta, r} \lp \frac{I(4N \tau, \eta)}{2\pi i} \rp
= \sum_{d > 0, d \equiv 0, 3\bmod{4}} \Tr_{N, \Delta, d}(\eta) q^d + \Theta^*_\Delta(\tau, \eta) \in \Ac_{3/2}(\Gamma_0(4N)),
$$ 
where 
\begin{equation}
  \label{eq:Theta*}
  \begin{split}
      \Theta^*_\Delta(\tau, \eta) := &
 \sum_{\begin{subarray}{c} d > 0\\ d \equiv 0, 3\bmod{4} \end{subarray}} q^d
\sum_{\zeta \in \Gamma_0(N) \backslash \Hb} r_\zeta(\eta) 
 \sum_{X \in \Qc_{N, -\Delta d} }  \frac{ \chi_\Delta(X) \sgn(\mathrm{d}(X, \zeta))}{2} \erfc \lp  \sqrt{ \frac{4 \pi v}{ |\Delta|}} |\mathrm{d}(X, \zeta)| \rp\\
&  +\sum_{[\ell] \in \Gamma_0(N) \backslash \Pb^1(\Qb)} r_{[\ell]}(\eta) \sum_{\delta \in L'/L_\Delta} \chi_\Delta(\delta) \Bb_1 \lp \frac{\delta_\ell}{|\Delta| \beta_\ell}  \rp.
  \end{split}
\end{equation}
Under the lowering operator $L_\tau$, the holomorphic part of $I_{\Delta, N}(\tau, \eta)$ vanishes and the non-holomorphic part becomes
\begin{equation}
  \label{eq:ThetaDelta}
  \Theta_\Delta(\tau, \eta) := 
- \sqrt{\frac{ v^3}{|\Delta|}}
\sum_{\zeta \in \Gamma_0(N) \backslash \Hb} r_\zeta(\eta) 
 \sum_{\begin{subarray}{c} d > 0\\ d \equiv 0, 3\bmod{4} \end{subarray}} q^d
 \sum_{X \in \Qc_{N, -\Delta d} }  \chi_\Delta(X) 
\mathrm{d}(X, \zeta)
 e^{  - 4 \pi v \mathrm{d}(X, \zeta)^2/ |\Delta|}.
\end{equation}
When $N = 1$ and $\eta = \frac{j'(z)}{j(z) - 1728} dz$, the series above simplifies to $v^{3/2} \overline{\theta_\Delta(\tau)}$ with $\theta_\Delta$ the Siegel theta function in Theorem \ref{thm:scalar1}.
The constant term of $\Theta^*_\Delta(\tau, \eta)$ can also be simplified. Since $N = 1$, $\Gamma_0(N) \backslash \Pb^1(\Qb)$ contains just the cusp $\ell = \ell_0 = \Rb \cdot \smat{0}{1}{0}{0}$, where $\eta$ has residue 1. The constant term then boils down to the sum
$$
\sum_{\delta \in L'/L_\Delta} \chi_\Delta(\delta) \Bb_1 \lp \frac{\delta_\ell}{|\Delta| }  \rp.
$$
For $\delta \in L'/L_\Delta$, $\ell \cap (L_\Delta + \delta)$ is non-trivial if and only if $\delta$ represents $\smat{0}{-C}{0}{0}$, in which case $\delta_\ell = -C$ and $\chi_\Delta(\delta) = \chi_\Delta([0, 0, C]) = \lp \frac{\Delta}{C} \rp$. Therefore, the sum above becomes $\sum_{C \in \Zb/|\Delta|\Zb} \lp \frac{\Delta}{C} \rp \Bb_1 \lp \frac{-C}{|\Delta|} \rp = L(0, \Delta)$ by Theorem 4.2 in \cite{Washington97}.

\section{Indefinite theta functions}

In his thesis \cite{ZwThesis}, Zwegers gave three ways to complete the mock theta functions of Ramanujan into a real-analytic modular object. One of the ways is through the use a real-analytic theta function $\vartheta_L^{c_1, c_2}(\tau) = \sum_{ h \in L'/L} \ef_h \vartheta_{L,h}^{c_1, c_2}(\tau)$ constructed from a lattice $L$ of signature $(p, 1)$ and two negative vectors $c_1, c_2$ such that $(c_1,c_2) <0$. Then $\vartheta_L^{c_1, c_2}(\tau)$ is a real-analytic modular form for weight $(p+1)/2$ on $\Mp_2(\Zb)$ with respect to the Weil representation $\rho_L$. 

In \cite{FK17} Theorem 4.7, Kudla und the second author gave another interpretation of $\vartheta_L^{c_1, c_2}$ as the integral of the Kudla-Millson theta form $\Theta(\tau, z, \varphi_\KM)$ for signature $(p,1)$. 
Namely, the conditions on $c_1$ and $c_2$ imply that they define points $[c_1]$ and $[c_2]$ in the same component of the symmetric space associated to $V = L \otimes_\Zb \Qb$. Then for the geodesic 
$$
\Dc_{c_1, c_2} = \{ \mathrm{span}(tc_1 + (1-t)c_2): t \in [0, 1]\}
$$
in hyperbolic space connecting the points $[c_1]$ and $[c_2]$, we have 
\begin{align}\label{Z-FK}
 \vartheta_{L,h}^{c_1, c_2}(\tau) = \int_{\Dc_{c_1, c_2}} \Theta_h(\tau, z, \varphi_{\KM}).
 \end{align}
Furthermore,  
\begin{align}
  \label{eq:vartheta_L}
  \vartheta_{L,h}^{c_1, c_2}(\tau) &= \sum_{m \ge 0} \sum_{X \in L_{m, h}} \frac{\sgn((X, c_1)) - \sgn((X, c_2))}{2} q^m \\ &\qquad \qquad + \sum_{m \in \Qb}  (\Theta_{m,h}(\tau, [c_2], \tpsi) - \Theta_{m,h}(\tau,[c_1], \tpsi)q^m.  \notag
\end{align}
Here $\tpsi$ is the singular form defined in \eqref{eq:tpsi} (extended to signature $(p,1)$). Moreover, the holomorphic coefficients $\half  (\sgn(X, c_1) - \sgn(X, c_2) )$ are the signed intersection number of $\Dc_{c_1, c_2}$ with the cycle $c_X$. Combing the considerations in \cite{FK17} with those in \cite{FM02} (or using Section~\ref{Boundary-tpsi}) one sees that \eqref{Z-FK} extends to the case when $c_1$ or $c_2$ (or both) are rational isotropic vectors in $V$. Here we use the convention \eqref{eq:tpsicusp}. For similar integrals of the Kudla-Millson theta series, also see \cite{Kudla17} (signature $(p, 2)$) and \cite{FK17b} (signature $(p, q)$). 

Now consider signature $(2, 1)$ and let $\eta = \eta_{\DD}$ be the differential of the third kind associated to the divisor $\DD := [\zeta_1] - [\zeta_2] \in \Div^0(M^\ast)$. Here $\zeta_j \in M^\ast$ is the image of $c_j$ for $j = 1, 2$ in $M^\ast$.
Let $\gamma_{} \subset M$ be the image of $\Dc_{c_1, c_2}$. 
When summing the signed intersection of $c_X$ and $\Dc_{c_1, c_2}$ over $X \in L_{m, h}$, we obtain a $\Gamma$-invariant quantity, which equals to the signed intersection $\Ib(\gamma, [c(X)])$ (defined in \S \ref{subsec:differential}) averaged over $X \in \Gamma \backslash L_{m, h}$.

It is clear from this definition and Theorem \ref{thm:FE} that $-2 \pi i \vartheta^{c_1, c_2}_L(\tau)$ and $I(\tau, \eta_{})$ have the same non-holomorphic part. The discrepancy in their holomorphic part can be explained geometrically using the Kudla-Millson lift of a closed 1-form.
By Lemma \ref{lemma:intersection}, we can relate $\vartheta^{c_1, c_2}_L(\tau)$ to $I(\tau, \eta)$ as follows.
\begin{prop}
  \label{prop:Zwegers}
Fix a section $s: H_1(M^*, \Zb) \to H_1(M^*\backslash \{\zeta_1, \zeta_2\}, \Zb)$ of the projection $\pi$ and $\gamma_s$ a path from $\zeta_1$ to $\zeta_2$ as in \S \ref{subsec:differential}.
Let $\omega_{\eta, s} \in H^1_{\mathrm{DR}}(M^*)$ be the closed differential form satisfying  \eqref{eq:omega_eta}. 
Then there exists a unary theta series $g_{ \eta, s}\in S_{3/2, \rho_L}(\SL_2(\Zb))$ such that
\begin{equation}
  \label{eq:Zwegers2}
 \vartheta^{c_1, c_2}_{L}(\tau) = - \frac{1}{2\pi i} \lp   I(\tau, \eta) - I(\tau, \omega_{\eta, s}) + g_{ \eta, s}(\tau) \rp 
+ I(\tau, \omega_{[\gamma - \gamma_s]}),
\end{equation}
where $\omega_{\mathbf{c}}$ denotes the Poincar\'{e} dual of $\mathbf{c} \in H_1(M^*, \Zb)$. 
\end{prop}

\begin{proof}
From the definition of $\vartheta^{c_1, c_2}_L$ and Theorem \ref{thm:FE}, we see that both sides have the same non-holomorphic part. 
For the holomorphic part, suppose $c(X)$ is closed for $X \in L_{m, h}$.
This happens when $\sqrt{mN} \not\in \Qb^\times$ or $M^*$ has only one cusp.
We need to show 
$$
\sum_{X \in \Gamma_L \backslash L_{m, h}} \Ib(\gamma, [c(X)])
- \int_{c(X)} \omega_{[\gamma - \gamma_s]}
=
-\frac{1}{2\pi i} 
\sum_{X \in \Gamma_L \backslash L_{m, h}}
\int_{c(X)} \eta - \omega_{\eta, s}
$$
for all $h \in L'/L$.
By definition of the signed intersection number, we know that $\int_{c(X)} \omega_{[\gamma - \gamma_s]} = \Ib(\gamma, [c(X)]) - \Ib(\gamma_s, [c(X)])$. Lemma \ref{lemma:intersection} then finishes the proof.
\end{proof}

\begin{rmk}
If $M^*$ has genus zero, resp.\ only one cusp, then $\omega_{\eta, s}$, resp.\ $g_{ \eta, s}$, vanishes identically. 
Furthermore, if $M^{\ast}$ is the modular curve as in the introduction, we obtain the equality of the twisted theta integral over $\eta$ with the (twisted) $\vartheta^{c_1, c_2}$. 
\end{rmk}

\section{Shimura Curve and Mock Theta Function.}

In this section, we will give an example where $\eta$ is a differential of the third kind on a Shimura curve $M$ and the generating series $I(\tau, \eta)$ is closely related to the third order mock theta function of Ramanujan.

Let $B$ be a quaternion algebra over $\Qb$ with discriminant $6$ \cite[\textsection 3]{Voight09}. It is the $\Qb$-algebra generated by $\alpha_0, \alpha_1$ satisfying
\begin{equation}
  \label{eq:ab}
  \alpha^2_0 = -1,  \alpha^2_1 = 3, \alpha_2 := \alpha_0 \alpha_1 = - \alpha_1 \alpha_0.
\end{equation}
As before, we let $Q = -\Nm$ be the quadratic form.
The subspace $B^0 \subset B$ of trace zero elements is given by $\Qb \alpha_0 \oplus \Qb \alpha_1 \oplus \Qb \alpha_2$. 
It is more convenient to view $B$ as a $\Qb$-subalgebra of $M_2(\Rb)$ via the embedding
\begin{equation}
  \label{eq:iota_inf}
  \begin{split}
    \iota_\infty: B &\hookrightarrow M_2(\Rb) \\
\alpha_0, \alpha_1, \alpha_2 &\mapsto  \pmat{0}{-1}{1}{0}, \sqrt{3} \pmat{1}{0}{0}{-1}, \sqrt{3} \pmat{0}{1}{1}{0}.
  \end{split}
\end{equation}
This extends to an isometry between quadratic spaces $(B_\Rb, \Nm)$ and $(M_2(\Rb), \det)$.
We slightly abuse the notation by using $X$ to denote $\iota_\infty(X)$ for $X \in B_\Rb$.

We choose a maximal order $\Oc := \Zb \oplus \Zb \alpha_0 \oplus \Zb \alpha_1 \oplus \Zb \alpha_3$ with $\alpha_3 := (1 + \alpha_0 + \alpha_1 + \alpha_2)/2$ and the units $\Oc^*_1 \subset \Oc$ of reduced norm 1 act by conjugation on $\Oc$. 
We take as our lattice
\begin{equation}
  \label{eq:L_compact}
L := B^0 \cap \Oc = \Zb \alpha_0 \oplus \Zb \alpha_1 \oplus \Zb \alpha_2.
\end{equation}
Then $Q(\alpha_0) = -1, Q(\alpha_1) = Q(\alpha_2) = 3$ and 
\begin{equation}
\label{eq:Ldual}
L' = \frac{1}{2} \Zb \alpha_0 \oplus \frac{1}{6} \Zb \alpha_1 \oplus  \frac{1}{6} \Zb \alpha_2, \, L'/L = \bigoplus_{j = 1}^3 L_j'/L_j, \, L_j := \Zb \alpha_j.
\end{equation}
We identify $L'_0/L_0$ with $\half\Zb/\Zb$ and $L'_1/L_1, L'_2/L_2$ with $\frac{1}{6} \Zb/\Zb$.
The group $\Oc^*_1 \subset \SO(L)$ fixes the connected component of $\Dc$ and can be viewed as a discrete subgroup of $\SL_2(\Rb)$ via the map $\iota_\infty$ above. 
It can be generated by $\alpha_0, \alpha_0 + \alpha_3, 2\alpha_0 + \alpha_2$, and naturally acts on $L'/L$ with the image in $\Aut(L'/L)$ is isomorphic to $\Zb/6\Zb$.
Let $\Gamma := \Oc^*_1 \cap \Gamma_L$ be the index 6 subgroup of $\Oc^*_1$.
The quotient $M = \Gamma \backslash \Hb$ is a Shimura curve of genus one.

Let $\eta$ be a differential of the third kind on $M$.
For $m \in \Qb_{}$ and $h \in L'/L$, the quantity $\Tr_{m, h}(\eta)$ is defined in  \eqref{eq:trace}.
Theorem \ref{thm:FE} then implies that the generating series 
\begin{equation}
\label{eq:trace_gen}
\sum_{h \in L'/L} \ef_h \sum_{m \in \Qb} \Tr_{m, h}(\eta) q^m
\end{equation}
is the holomorphic part of a mixed mock modular form with shadow $\sum_{\zeta \in M} r_\zeta(\eta) \vartheta(\tau, \zeta)$, where $\res(\eta) = \sum_{\zeta \in M} r_\zeta(\eta) [\zeta]$ and
\begin{align*}
\vartheta(\tau, z) &:= v \sum_{h \in L'/L} \ef_h \vartheta_h(\tau, z), \;
\vartheta_h(\tau, z) := \sum_{X \in L + h} \frac{(X, X(z))}{2} \ebf \lp \frac{(X, X(z))^2}{4} \tau 
-
\frac{|(X, X^\perp(z))|^2}{4} \overline{\tau} \rp
\end{align*}
for $z = x + iy \in \Hb$, $X(z) = \frac{1}{y} \smat{-x}{|z|^2}{-1}{x}$ and $X^\perp(z) = \frac{1}{y}\smat{-z}{z^2}{-1}{z}$.

Let $z_0 = x_0 + iy_0 := \frac{1 + \sqrt{-2}}{\sqrt{3}} \in \Hb$ be a CM point. Then 
\begin{equation}
  \label{eq:intersection}
  \begin{split}
      \Rb X(z_0) \cap L &= \Zb \lambda_0,   \; ( \Rb \Re(X^\perp(z_0)) \oplus \Rb \Im(X^\perp(z_0))) \cap L = \Zb \lambda_1 \oplus \Zb \alpha_2, \\
\lambda_0 &:= 3 \alpha_0 + \alpha_1, \; \lambda_1 := \alpha_0 + \alpha_1,
  \end{split}
\end{equation}
and $\tilde{L} := \Zb \lambda_0 \oplus \Zb \lambda_1 \oplus \Zb \alpha_2$ is a sublattice of $L$ of index 2. 
Denote the lattices $\Zb \lambda_0, \Zb \lambda_1$ by $N, P$ respectively.
Then $\tilde{L} = N \oplus P \oplus L_2$ and $\tilde{L}'/\tilde{L} = N'/N \oplus P'/P \oplus L_2'/L_2 \subset (\Qb/\Zb)^3$.
For $X = r_0 \lambda_0 + r_1 \lambda_1 + r_2 \alpha_2 \in V$, it is easy to check that 
\begin{equation}
\label{eq:3}
 ( X, X(z_0)) = -\frac{2r_0}{\sqrt{6}} Q(\lambda_0) = 2\sqrt{6} r_0 , \;
|(X, X^\perp(z_0))|^2 =  8 r_1^2 + 12 r_2^2.
\end{equation}
Suppose $X \in \tilde{L}'$, then $X \in L'$ if and only if $3 r_0 + r_1 \in \half \Zb $. This also implies $r_0 + r_1 \in \frac{1}{6} \Zb$.
Therefore, we have the following two-to-one surjective map
\begin{equation}
  \label{eq:phi}
  \begin{split}
    \phi: L'/\tilde{L} & \to L'/L \\
\mu = (\mu_0, \mu_1, \mu_2) &\mapsto (3 \mu_0 + \mu_1, \mu_0 + \mu_1, \mu_2)
  \end{split}
\end{equation}
where the addition is carried out in $\Qb/\Zb$. 
It is easy to check that if $X \in (\tilde{L} + \mu) \cap L' \subset V$, then $X \in L + \phi(\mu)$. 
This map induces a linear map $A_\phi: \Cb[L'/L] \to \Cb[\tilde{L}'/\tilde{L}]$ defined by
\begin{equation}
  \label{eq:Aphi}
  A_\phi(\ef_h) := \sum_{\mu \in \phi^{-1}(h)} \ef_\mu
\end{equation}
for each $h \in L'/L$.

Define the theta functions
\begin{equation}
  \label{eq:thetas}
  \begin{split}
      \vartheta_{N}(\tau) &:= \sum_{h \in N'/N} \ef_{h} \sum_{\lambda \in N + h} \frac{(\lambda, \lambda_0)}{\sqrt{6}} \ebf(Q(\lambda)\tau), \,
  \theta_{P}(\tau) := \sum_{h \in P'/P} \ef_{h} \sum_{\lambda \in P + h} \ebf(Q(\lambda)\tau), \\
\theta_{L_2}(\tau) &:= \sum_{h \in L_2'/L_2} \ef_h \sum_{\lambda \in L_2 + h} \ebf(Q(\lambda) \tau).
  \end{split}
\end{equation}
Notice that $\vartheta_{N, \mu_0}(\tau) = -\vartheta_{N, -\mu_0}(\tau)$ and $\theta_{P, \mu_1}(\tau) = \theta_{P, -\mu_1}(\tau)$ for $\mu_0 \in N'/N, \mu_1 \in P'/P$.
For each $h = (h_0, h_1, h_2) \in L'/L$, the function $\vartheta_h(\tau, z_0)$ can be written as
$$
\vartheta_{h} (\tau, z_0) = v \sum_{\mu = (\mu_0, \mu_1, \mu_2) \in L'/\tilde{L}, \; \phi(\mu) = h}
 \vartheta_{N, \mu_0} (\tau) \theta_{P, \mu_1}(-\overline{\tau}) \theta_{L_2, \mu_2}(-\overline{\tau}) .
$$
The same argument applies when $z_0$ is replaced by $- \overline{z_0}$. In that case, the equation above becomes
$$
\vartheta_{h} (\tau, - \overline{z_0}) = v \sum_{\mu = (\mu_0, \mu_1, \mu_2) \in L'/\tilde{L}, \; \phi(\mu) = h}
 \vartheta_{N, \sigma(\mu_0)} (\tau) \theta_{P, \mu_1}(-\overline{\tau}) \theta_{L_2, \mu_2}(-\overline{\tau}),
$$
where $\sigma \in \mathrm{Iso}(N'/N)$ sends $\mu_0 \in N'/N = \frac{1}{12} \Zb$ to $5 \mu_0 \in N'/N$.
Since $\sigma$ is an isometry, it defines an automorphism on $\Ac_{k}(\rho_N)$, which we also use $\sigma$ to denote.

 Let $\eta$ a differential of the third kind on $M$ with the residue divisor 
\begin{equation}
  \label{eq:etaz0}
  \res(\eta) = [z_0 ] - [-\overline{z_0}].
\end{equation}
Then the generating series  \eqref{eq:trace_gen} is the holomorphic part of an automorphic form in $\Ac_{3/2}(\rho_L)$, whose image under $\xi_{3/2}$ is 
\begin{equation}
  \label{eq:shadow1}
  v A_\phi^*((\vartheta_{N, \sigma}(\tau)) \otimes \theta_P(-\overline{\tau}) \otimes \theta_P(-\overline{\tau})) \in \Ac_{1/2}(\overline{\rho_L}),
\end{equation}
where $\vartheta_{N, \sigma}(\tau) : = \vartheta_{N}(\tau) - \sigma(\vartheta_N(\tau))$.

On the other hand, recall that the following third order mock theta functions of Ramanujan
\begin{equation}
  \label{eq:mock3}
  \begin{split}
      f(q) &:= \sum_{n = 0}^\infty \frac{q^{n^2}}{\prod^{n}_{k=1}(1+q^k)^2}, 
\omega(q) :=  \sum_{n = 0}^\infty \frac{q^{2n^2 + 2n}}{\prod^n_{k=1}(1 - q^{2k - 1})^2}.
  \end{split}
\end{equation}
If one further set $q = \ebf(\tau)$ and define 
\begin{equation}
  \label{eq:fj}
  f_0(\tau) := q^{-1/24} f(q), \; f_1(\tau) := 2q^{1/3} \omega(q^{1/2}), \; f_2 := 2q^{1/3} \omega(-q^{1/2}),
\end{equation}
then Zwegers showed that the column vector $(f_0, f_1, f_2)^T$ is a mock modular form with a unary theta function as its shadow \cite{Zw00}. 
One can use this to show that the vector (see \cite{BS17}) 
$$
\tilde{\vartheta}_{N, \sigma}^+ := -\frac{1}{4}(0, f_0, f_2 - f_1, 0, -f_1 - f_2, -f_0, 0, f_0, f_1 + f_2, 0, f_1 - f_2, -f_0)^T
$$
is the holomorphic part of a harmonic Maass form $\tilde{\vartheta}_{N, \sigma}$, whose image under $\xi_{3/2}$ is $\vartheta_{N, \sigma}$. 
This has a pole of the form $q^{-1/24} + O(1)$ at the components $\ef_{\mu_0}$ for $\mu_0  \in \{ \tfrac{1}{12}, \tfrac{5}{12}, \tfrac{7}{12}, \tfrac{11}{12}\} \subset N'/N$. In these cases, $\mu_1 \in \{\tfrac{1}{4}, \tfrac{3}{4} \} \subset P'/P$ in order for $(\mu_0, \mu_1, \mu_2) \in L'/\tilde{L}$. For these $\mu_1$, the function $\theta_{P, \mu_1}(\tau)$ has the form $q^{1/8} + O(q)$. Therefore for any $h \in L'/L$, the function 
$$
\sum_{\mu = (\mu_0, \mu_1, \mu_2) \in L'/\tilde{L}, \phi(\mu) = h} \tilde{\vartheta}_{N, \sigma, \mu_0}(\tau) \theta_{P, \mu_1}(\tau) \theta_{L_2, \mu_2}(\tau)
$$
decays exponentially at the cusp. 
Therefore, there exists a cusp form $g \in S_{3/2}(\rho_L)$ such that 
$$
\frac{I(\tau, \eta)}{2\pi i} = A^*_\phi(\tilde{\vartheta}_{N, \sigma}(\tau) \otimes \theta_P(\tau) \otimes \theta_P({\tau}) ) + g(\tau).
$$

\bibliography{ShLift.bib}{}
\bibliographystyle{amsplain}

\end{document}